\documentclass[12pt]{amsart}
\usepackage[utf8]{inputenc}
\usepackage[english]{babel}
\usepackage{graphicx}
\usepackage{amsmath}
\usepackage{amssymb}
\usepackage{amscd}
\usepackage{amsthm}
\usepackage{amscd}

\usepackage[matrix,arrow,curve]{xy}
\usepackage{stmaryrd}
\usepackage{hyperref}

\usepackage[table]{xcolor}

\usepackage{tikz}
\usetikzlibrary{arrows.meta}
\usetikzlibrary{bending}

\usepackage{graphicx}
\usepackage{xspace}
\usepackage{enumerate}
\usepackage{verbatim}
\usepackage{tikz-cd}

\usepackage{geometry} 
\DeclareMathOperator{\F}{\mathbb{F}}
\DeclareMathOperator{\CH}{\mathrm{CH}}

\address{Nikita Geldhauser:
Ludwig-Maximilians-Universit\"at M\"unchen, The\-re\-sien\-str. 39, 80333 M\"unchen, Germany}
\email{geldhauser@math.lmu.de}

\address{Andrei Lavrenov, Ludwig-Maximilians-Universit\"at M\"unchen, The\-re\-sien\-str. 39, 80333 M\"unchen, Germany}
\email{avlavrenov@gmail.com}

\address{Victor Petrov,  Laboratory of Modern Algebra and Applications, St. Petersburg State University, 14th Line V.O. 29b, 199178 St. Petersburg, Russia 
and 
PDMI RAS, Fontanka emb. 27, 191023 St. Petersburg, Russia}
\email{victorapetrov@googlemail.com}

\address{Pavel Sechin, Universität Regensburg, Universitätsstr. 31, 93053 Regensburg, Germany}
\email{pavel.sechin@ur.de}



\begin{document}

\newcommand{\Sm}{\mathcal{S}^{\!}\mathsf{m}_k}
\newcommand{\Rings}{\mathcal{R}^{\!}\mathsf{ings}^*}
\newcommand{\rings}{\mathcal{R}^{\!}\mathsf{ings}}
\newcommand{\SmE}[1]{\mathcal{S}^{\!}\mathsf{m}_{#1}}
\newcommand{\Mot}[1]{\mathcal M^{\!}\mathsf{ot}_{#1}}
\newcommand{\Corr}[1]{\mathcal C^{\!}\mathsf{orr}_{#1}}
\newcommand{\MotF}[1]{\mathcal M_{#1}}
\newcommand{\QG}{\mathbb Q\Gamma}
\newcommand{\Inv}{\mathrm{Inv}(\QG)}
\newcommand{\Gal}{\mathrm{Gal}(k^{\mathrm{sep}}/k)}
\newcommand{\End}{\mathrm{End}}
\newcommand{\M}[1]{\mathcal{M}_{#1}}
\newcommand{\EG}{\!\,_EG}
\newcommand{\EGP}{\!\,_E(G/P)}
\newcommand{\EX}{\!\,_EX}
\newcommand{\XG}{\!\,_{\xi}G}
\newcommand{\XGP}{\!\,_{\xi}(G/P)}
\newcommand{\KQ}[1]{\mathrm K(n)^*\big(#1;\,\mathbb Q[v_n^{\pm1}]\big)}
\newcommand{\KZ}[1]{\mathrm K(n)^*\big(#1;\,\mathbb Z_{(p)}[v_n^{\pm1}]\big)}
\newcommand{\KZp}[1]{\mathrm K(n)^*\big(#1;\,\mathbb Z_p[v_n^{\pm1}]\big)}
\newcommand{\KF}[1]{\mathrm K(n)^*\big(#1;\,\mathbb F_p[v_n^{\pm1}]\big)}
\newcommand{\CHQ}[1]{\mathrm{CH}^*\big(#1;\,\mathbb Q[v_n^{\pm1}]\big)}
\newcommand{\KXZ}[1]{\!\,^{\mathrm K(n)\!}#1_{\,\mathbb Z_{(p)}[v_n^{\pm1}]}}
\newcommand{\KMotQ}{\Mot{\,\mathrm K(n)}}
\newcommand{\CHMotQv}{\Mot{\,\mathrm{CH}}}
\newcommand{\KXQ}[1]{\mathcal M_{\mathrm K(n)}(#1)}
\newcommand{\KMQ}{\mathcal M_{\,\mathrm K(n)}}
\newcommand{\CHMQv}{\MotF{\,\mathrm{CH}}}
\newcommand{\KCorrQ}{\Corr{\,\mathrm K(n)}}
\newcommand{\CHCorrQv}{\Corr{\,\mathrm{CH}}}
\newcommand{\CHCorrQ}{\Corr{\,\mathrm{CH}}}
\newcommand{\AMot}{\Mot A}

\newcommand{\e}{\varepsilon}
\newcommand{\con}{\ensuremath{\triangledown}}
\newcommand{\ra}{\ensuremath{\rightarrow}}
\newcommand{\tp}{\ensuremath{\otimes}}
\newcommand{\pr}{\ensuremath{\partial}}
\newcommand{\trigd}{\ensuremath{\triangledown}}
\newcommand{\dAB}{\ensuremath{\Omega_{A/B}}}
\newcommand{\QQ}{\ensuremath{\mathbb{Q}}\xspace}
\newcommand{\CC}{\ensuremath{\mathbb{C}}\xspace}
\newcommand{\RR}{\ensuremath{\mathbb{R}}\xspace}
\newcommand{\ZZ}{\ensuremath{\mathbb{Z}}\xspace}
\newcommand{\Zp}{\ensuremath{\mathbb{Z}_{(p)}}\xspace}
\newcommand{\Z}[1]{\ensuremath{\mathbb{Z}_{(#1)}}\xspace}
\newcommand{\NN}{\ensuremath{\mathbb{N}}\xspace}
\newcommand{\LL}{\ensuremath{\mathbb{L}}\xspace}
\newcommand{\inN}{\ensuremath{\in\mathbb{N}}\xspace}
\newcommand{\inQ}{\ensuremath{\in\mathbb{Q}}\xspace}
\newcommand{\inR}{\ensuremath{\in\mathbb{R}}\xspace}
\newcommand{\inC}{\ensuremath{\in\mathbb{C}}\xspace}
\newcommand{\OO}{\ensuremath{\mathcal{O}}\xspace}
\newcommand{\rarr}{\rightarrow}
\newcommand{\Rarr}{\Rightarrow}
\newcommand{\xrarr}[1]{\xrightarrow{#1}}
\newcommand{\larr}{\leftarrow}
\newcommand{\lrarr}{\leftrightarrows}
\newcommand{\rlarr}{\rightleftarrows}
\newcommand{\rrarr}{\rightrightarrows}
\newcommand{\al}{\alpha}
\newcommand{\bt}{\beta}
\newcommand{\ld}{\lambda}
\newcommand{\om}{\omega}
\newcommand{\Kd}[1]{\ensuremath{\Omega^{#1}}}
\newcommand{\KKd}{\ensuremath{\Omega^2}}
\newcommand{\vd}{\partial}
\newcommand{\PC}{\ensuremath{\mathbb{P}_1(\mathbb{C})}}
\newcommand{\PPC}{\ensuremath{\mathbb{P}_2(\mathbb{C})}}
\newcommand{\derz}{\ensuremath{\frac{\partial}{\partial z}}}
\newcommand{\derw}{\ensuremath{\frac{\partial}{\partial w}}}
\newcommand{\mb}[1]{\ensuremath{\mathbb{#1}}}
\newcommand{\mf}[1]{\ensuremath{\mathfrak{#1}}}
\newcommand{\mc}[1]{\ensuremath{\mathcal{#1}}}
\newcommand{\id}{\ensuremath{\mbox{id}}}
\newcommand{\dd}{\ensuremath{\delta}}
\newcommand{\bu}{\bullet}
\newcommand{\ot}{\otimes}
\newcommand{\boxt}{\boxtimes}
\newcommand{\op}{\oplus}
\newcommand{\mt}{\times}
\newcommand{\Gm}{\mathbb{G}_m}
\newcommand{\Ext}{\ensuremath{\mathrm{Ext}}}
\newcommand{\Tor}{\ensuremath{\mathrm{Tor}}}

\newcommand{\kn}[1]{\mathrm K(n)^*(#1)}
\newcommand{\Kn}[1]{\ensuremath{\mathrm K({#1})}}
\newcommand{\ckn}[1]{\mathrm{CK}(n)^*(#1)}
\newcommand{\grckn}[1]{\mathrm{gr}_\tau^{*}\,\mathrm{CK}(n)^{*}(#1)}
\newcommand{\so}{\mathrm{SO}_m}
\newcommand{\SO}[1]{\ensuremath{\mathrm{SO}_{#1}}}
\newcommand{\pt}{\mathrm{pt}}

\makeatletter
\newcommand{\colim@}[2]{
  \vtop{\m@th\ialign{##\cr
    \hfil$#1\operator@font colim$\hfil\cr
    \noalign{\nointerlineskip\kern1.5\ex@}#2\cr
    \noalign{\nointerlineskip\kern-\ex@}\cr}}
}
\newcommand{\colim}{
  \mathop{\mathpalette\colim@{\rightarrowfill@\textstyle}}\nmlimits@
}
\makeatother

\newtheorem{lm}{Lemma}[section]
\newtheorem{lm*}{Lemma}
\newtheorem*{tm*}{Theorem}
\newtheorem*{tms*}{Satz}
\newtheorem{tm}[lm]{Theorem}
\newtheorem{prop}[lm]{Proposition}
\newtheorem*{prop*}{Proposition}
\newtheorem{cl}[lm]{Corollary}
\newtheorem*{cor*}{Corollary}
\theoremstyle{remark}
\newtheorem*{rk*}{Remark}
\newtheorem*{rm*}{Remark}
\newtheorem{rk}[lm]{Remark}
\newtheorem*{xm}{Example}
\theoremstyle{definition}
\newtheorem{df}{Definition}
\newtheorem*{nt}{Notation}
\newtheorem{Def}[lm]{Definition}
\newtheorem*{Def-intro}{Definition}
\newtheorem{Rk}[lm]{Remark}
\newtheorem{Ex}[lm]{Example}

\theoremstyle{plain}
\newtheorem{Th}[lm]{Theorem}
\newtheorem*{Th*}{Theorem}
\newtheorem*{Th-intro}{Theorem}
\newtheorem{Prop}[lm]{Proposition}
\newtheorem*{Prop*}{Proposition}
\newtheorem{Cr}[lm]{Corollary}
\newtheorem{Lm}[lm]{Lemma}
\newtheorem*{Conj}{Syzygies Conjecture for Algebraic Cobordism}
\newtheorem*{BigTh}{Classification of Operations Theorem  (COT)}
\newtheorem*{BigTh-add}{Algebraic Classification of Additive Operations Theorem  (CAOT)}

\tikzcdset{
arrow style=tikz,
diagrams={>={Straight Barb[scale=0.8]}}
}

\title[Morava of orthogonal groups and motives of quadrics]{Morava K-theory of orthogonal groups and motives of projective quadrics}
\author{Nikita Geldhauser, Andrei Lavrenov,\\ Victor Petrov, Pavel Sechin}
\maketitle

\begin{abstract}
We compute the algebraic Morava K-theory ring of split special orthogonal and spin groups. In particular, we establish certain stabilization results for the Morava K-theory of special orthogonal and spin groups. Besides, we apply these results to study Morava motivic decompositions of orthogonal Grassmannians. For instance, we determine all indecomposable summands of the Morava motives of a generic quadric.
\end{abstract}

\section{Introduction}

One of the questions of algebraic geometry over fields that are not algebraically closed
is the study of projective homogeneous varieties.
An effective method is to investigate various generalized cohomology theories of these varieties, 
constructed, for example, via algebraic cycles, vector bundles or cobordism classes. 
If $G$ is an algebraic group that acts transitively on a projective homogeneous variety $X$,
then the value of the cohomology theory $A^*(G)$ of the group already 
contains substantial information about the value of the cohomology theory $A^*(X)$. 

In this article we examine $A^*=\Kn{n}^*$, the algebraic Morava K-theory for prime 2, and $G=\so$, the split special orthogonal group, i.e., the corresponding projective homogeneous varieties are quadrics and (higher) orthogonal Grassmannians. The computation of the ring $\Kn{n}^*(\so)$ (Theorem \ref{thm:intro_so}) is one of the main results of this article. Among other things this provides new tools for the study of Morava motives, and as an instance of this principle, we determine all indecomposable summands of the $\Kn{n}$-motive of a generic quadric $Q$ (Theorem~\ref{thm:intro_quad}).

\subsection{Overview of Morava K-theory}
Levine and Morel defined a universal oriented cohomology theory, called the algebraic cobordism (see \cite{LM}). 
It allows one to consider algebraic analogues of well studied topological oriented cohomology theories such as Morava K-theories over fields of characteristic $0$.

The algebraic cobordism theory of Levine--Morel and arbitrary oriented cohomology theories~\cite{LP,Lev,Vcob,PS,VY,NZ,CPZ,V12,Se1,GV,PShopf,Mer,Sm}
are extensively studied now.
 Any oriented cohomology theory is endowed with a formal group law, and any formal group law over any commutative ring comes from a certain oriented cohomology theory in this sense. Among various theories corresponding to the same formal group law there exists a universal one called {\it free}. The class of free theories contains algebraic cobordism, Chow groups, i.e., algebraic cycles modulo rational equivalence, Grothendieck's $\mathrm K^0$, and Morava K-theories. One of the major features of these theories for our purposes is the Rost nilpotence principle for projective homogeneous varieties recently proven in~\cite{GV}.

Vishik gave a geometric description of free theories in~\cite{V12}, which allowed him to construct operations on algebraic cobordism, and later the fourth author used Vishik's description to construct operations from Morava K-theories, see~\cite{Se1,Se2}. Vishik's results imply, in particular, that the category of free theories and multiplicative operations is equivalent to the category of $1$-dimensional commutative graded formal group laws~\cite[Theorem~6.9]{V12}. 

Working with quadratic forms, it is natural to consider localized at $2$ coefficients $\mathbb Z_{(2)}$ instead of integral, and, therefore, consider only formal group laws over $\mathbb Z_{(2)}$-algebras. Then by a theorem of Cartier we can restrict ourselves to a narrower class of formal group laws, called {\it $2$-typical} (every formal group law over a $\mathbb Z_{(2)}$-algebra is isomorphic to a $2$-typical one), in particular, it is natural to consider the {\it universal $2$-typical} formal group law. It admits a standard construction as a formal group law over the polynomial ring with infinite number of variables $\mathbb Z_{(2)}[v_1,v_2,\ldots]$ defined by recurrent identities, see, e.g.,~\cite[Appendix~A2]{Rav}. Then passing to $\mathbb F_2$, specifying $v_k=0$ for $k\neq n$, and inverting $v_n$, we obtain a formal group law over $\mathbb F_{2}[v_n^{\pm1}]$, and the corresponding free theory is called the Morava K-theory $\mathrm K(n)^*$.

It is worth remarking that Morava K-theories are related to higher powers of the fundamental ideal $I$ in the Witt ring and, more generally, to cohomological invariants of algebraic groups. The first and the fourth authors stated in~\cite{SeSe} a ``Guiding Principle'', which dates back to Voevodsky's program~\cite{VoevMor} on the proof of the Bloch--Kato conjecture, and specifies this relation. In particular, the $n$-th Morava motive of a quadric
depends only on its dimension and on the class of the quadratic form in the Witt ring modulo $I^{n+2}$~\cite[Proof of Proposition~6.18]{SeSe}.

Recently, it was shown by
the first and the fourth authors \cite{SeSe} that motives for the {\it Morava K-theory} allow to obtain new results on Chow groups of quadrics. Moreover, they most naturally fill in the gap between $\mathrm K^0$-motives
and Chow motives. In this article this work is continued and is brought to the next level.

\subsection{Overview of motives}\label{sec1.1}

The theory of motives envisioned by Grothendieck is supposed to bridge the study of cohomology of algebraic varieties with rational 
coefficients with the study of various geometric and arithmetic properties of the variety 
itself that are expressed via algebraic cycles.
 It turned out, however, that this theory is much more profound and fundamental 
 and also accounts for phenomena of cohomology with integral coefficients.
 For example, the Chow and $\mathrm K^0$-motives with integral and finite coefficients are nowadays a common tool 
for the study of projective homogeneous varieties over non-algebraically closed fields.

One of the most striking applications of motives was Voevodsky's proof~\cite{Voev,OVV} of the Milnor conjecture which relies on Rost's computation~\cite{Rost} of
the motive of a Pfister quadric. More generally, the structure of Chow motives of norm varieties plays a crucial role in the proof of the Bloch--Kato conjecture by Rost and Voevodsky (see~\cite{Vo11}).

Chow motives of projective quadrics and, more generally, of projective homogeneous varieties were studied by Brosnan, Chernousov, Gille, Karpenko, Merkurjev, Vishik, Zainoulline and many others,
and there exist plenty various applications of Chow motives to quadratic forms, and, more generally, to algebraic groups, including~\cite{Kholes,Kfirst,KRost,YSteen,KZh,VJinv,Vrost,PSjinv,CGM,Bdec,KarPS,PS10,Vuinv,Vexcel}.

Motivic decompositions of projective homogeneous varieties were used to construct a new cohomological invariant for groups of type $\mathrm{E}_8$ and
to solve a problem of Serre about their finite subgroups (see \cite{GS10} and \cite{S16}). Besides, Garibaldi, Geldhauser, and Petrov used decompositions of Chow motives to relate the rationality of some parabolic subgroups of groups of type $\mathrm{E}_7$ with the Rost invariant, proving a conjecture of Rost and solving a question of Springer in \cite{GPSshells}.

A new tool for the study of motives of homogeneous varieties was not long ago developed by the first and by the third author \cite{PShopf}.
Let $G$ be a split semisimple algebraic group over a field and let $A^*$ be an oriented cohomology theory in the sense of Levine--Morel.
Then there exists a functor from the category of $A^*$-motives of twisted forms of smooth projective cellular $G$-varieties
to the category of graded comodules over a bi-algebra obtained as the quotient of the Hopf algebra $A^*(G)$ modulo certain bi-ideal.

In the case of twisted flag varieties this quotient is called the {\it $J$-invariant}. Note that in the case of a generic twisted flag variety the $J$-invariant is just the Hopf algebra $A^*(G)$ itself. One of the main properties of the $J$-invariant is that it carries full information 
on the motivic decomposition type of {\it generically split} twisted flag varieties (see \cite[Theorem~5.7]{PShopf}).

We remark finally that the $J$-invariant was introduced previously in \cite{PSjinv} and \cite{VJinv} for the Chow theory, and it was an essential ingredient in the solution of problems mentioned above and, in particular, in the solution of a problem of Serre about groups of type $\mathrm E_8$ and its finite subgroups.

\subsection{Overview of results}

\subsubsection{Morava K-theory of orthogonal groups}

 The main result of the present article is a computation of the ring of the Morava K-theory for the prime 2 
 of split special orthogonal groups. 

The investigation of the structure of oriented cohomology theories of split algebraic groups has a long history (see, for example, \cite{Le93} and \cite{Me97} for $\mathrm K^0$ or \cite[Table~2]{Kac} for the Chow theory). For example, the structure of the Chow ring  $\mathrm{CH}^*\big(\mathrm{SO}_m)$
is explicitly computed by Marlin \cite{Ma74}. One can write a closed formula for $\CH^*(\so)/2$ as
\begin{equation}\label{eq:ring}
     \mathbb F_2 [e_1, e_3,\ldots, e_{2r-1}]\Big/\Big(e_{2i-1}^{2^{k_i}},\, i=1\ldots r\Big), \quad \deg e_{2i-1} = 2i-1
\end{equation}
with explicit expressions for $k_i$ and $r$ depending on $m$ (see \cite{Kac}).

The following theorem extends the above result to the case of the Morava K-theory.

\begin{tm}[Theorem~\ref{answer}]\label{thm:intro_so}

As an $\mathbb{F}_2[v_n^{\pm 1}]$-algebra the Morava K-theory of split orthogonal groups can be described explicitly as follows:

$$
\mathrm K(n)^*(\mathrm{SO}_{m})\cong\mathbb F_2[v_n^{\pm1}][e_1,e_3,\ldots,e_{2r'-1}]\Big/\Big(e_{2i-1}^{2^{k'_i}},\, i=1\ldots r'\Big),
$$
where $r'=\mathrm{min}\left(2^{n-1},\,\left\lfloor\frac{m+1}{4}\right\rfloor\right)$ and $$k'_i=\mathrm{min}\left(\left\lfloor\mathrm{log}_2\left(\frac{2^{n+1}-1}{2i-1}\right)\right\rfloor,\,\left\lfloor\mathrm{log}_2\left(\frac{m-1}{2i-1}\right)\right\rfloor\right).$$
\end{tm}

We remark that the {\it topological} Morava K-theory of 
the connected component of the real compact Lie group of orthogonal matrices $\mathrm{SO}(m)$ is known only additively~\cite{Nis,Rao}.

The proof of the above theorem is very different from~\cite{Ma74,Kac}, although it relies on~(\ref{eq:ring}). It is obtained by combining the following three claims
that might be of independent interest.

\begin{enumerate}[(i)]
    \item {\it Stabilization.}
   
   \noindent 
            The canonical pullback map along the natural inclusion $$\mathrm K(n)^*\big(\mathrm{SO}_m \big)\rightarrow\mathrm K(n)^*\big(\mathrm{SO}_{m-2}\big)$$ is an isomorphism for $m\ge 2^{n+1}+1$.
\end{enumerate}
We remark that for the topological Morava K-theory the natural map $\mathrm{K}(n)^*_{\mathrm{top}}(\mathrm{SO}(m))\rightarrow\mathrm{K}(n)^*_{\mathrm{top}}(\mathrm{SO}(m-2))$ is {\it not} an isomorphism~\cite[Corollary~2.9]{Nis}.

Recall that for an arbitrary smooth variety $X$ there exists a canonical surjective ring homomorphism $$\rho\colon\CH^*(X)\ot \mathbb F_2[v_n,v_n^{-1}]\rarr \mathrm{gr}^*_\tau \kn{X}$$
where $\tau$ denotes the topological filtration on the Morava K-theory. In general, $\rho$ is an isomorphism only up to degree $2^n$~\cite{Se1}. However, for $X=\mathrm{SO}_m$ we can prove the following result.

\begin{enumerate}[(ii)]
    \item {\it Unstable coincidence of the associated graded ring  with Chow.}
         
         \noindent
         If $m\le 2^{n+1}$, then the canonical ring homomorphism
         $$\CH^*(\so)\ot \mathbb F_2[v_n,v_n^{-1}]\rarr \mathrm{gr}^*_\tau \kn{\so}$$
         is an isomorphism.
\end{enumerate}

Finally, we show that the Hopf algebra structure on $\kn{\so}$ allows to reconstruct
the algebra structure of $\kn{\so}$ from that of $\mathrm{gr}^*_\tau \kn{\so}$. Here we use~(\ref{eq:ring}) and the explicit formulae for $k_i$.

\begin{enumerate}[(iii)]
    \item {\it Reconstruction of the ring structure.} 
    
  \noindent
For $m\le 2^{n+1}$  the ring $\kn{\so}$ is non-canonically isomorphic to the ring ${\mathrm{CH}^*(\so)\ot \mathbb F_2[v_n^{\pm 1}]}$. 
\end{enumerate}

The three claims described above are proved by different methods. The proof of the stabilization result relies on Demazure differential operators (see \cite{BGG,Dem,CPZ}),
the unstable coincidence is based on the structure of algebraic cobordism studied by Vishik \cite{Vcob}
and the last statement as already mentioned requires a subtle study of Hopf algebras.

The following table summarizes the results of the computation for a few small numbers $n$ and $m$
(we set $v_n=1$ to simplify the formulae; grey cells indicate the stable range):
\begin{center}
\begin{tabular}{ l  | c  c  c  c }
     & \SO{3},\,\SO{4}  &  \SO{5},\,\SO{6} & \SO{7},\,\SO{8} & \SO{9},\,\SO{10} \\
     \hline 
     \hline 

\Kn{1}  & \cellcolor[gray]{0.9} $\mathbb F_2 [e_1]/e_1^2$ & \cellcolor[gray]{0.9} $\mathbb F_2 [e_1]/e_1^2$ & \cellcolor[gray]{0.9} $\mathbb F_2 [e_1]/e_1^2$ & 
\cellcolor[gray]{0.9} $\mathbb F_2 [e_1]/e_1^2$ \\  
        \hline 
\Kn{2}   & $\mathbb F_2 [e_1]/e_1^2$ & $\mathbb F_2 [e_1]/e_1^4$ & \cellcolor[gray]{0.9} $\mathbb F_2 [e_1, e_3]/(e_1^4, e_3^2)$ & 
\cellcolor[gray]{0.9} $\mathbb F_2 [e_1, e_3]/(e_1^4, e_3^2)$ \\
     \hline 
\Kn{3}   & $\mathbb F_2 [e_1]/e_1^2$ & $\mathbb F_2 [e_1]/e_1^4$ & $\mathbb F_2 [e_1, e_3]/(e_1^4, e_3^2)$ & 
$\mathbb F_2 [e_1, e_3]/(e_1^8, e_3^2)$ \\   
    \hline 
    \hline 
$\mathrm{CH/2}$  & $\mathbb F_2 [e_1]/e_1^2$ & $\mathbb F_2 [e_1]/e_1^4$ & $\mathbb F_2 [e_1, e_3]/(e_1^4, e_3^2)$ & 
$\mathbb F_2 [e_1, e_3]/(e_1^8, e_3^2)$ \\ 
\hline 
\end{tabular}
\end{center}

\medskip

Observe that the cell ``$\mathrm K(n)^*(\mathrm{SO}_{2^{n+1}-1})$'' is the first stable cell both in the corresponding row and the column. Although stabilization ``in columns'' is a general phenomenon, for an arbitrary $X$ it happens only for $\mathrm K(n)^*(X)$ with $2^n>\mathrm{dim}\,X$. We underline again that the stabilization ``in rows'' is a special phenomenon for $\mathrm K(n)^*(\mathrm{SO}_m)$, it does not hold for $\mathrm K(n)^*_{\mathrm{top}}(\mathrm{SO}({m}))$.
This miraculous complementarity of two stabilizations is what allowed us to prove the main theorem.

\subsubsection{Applications to motives}

In view of the discussion in Section~\ref{sec1.1}, the above theorem establishes the ring structure of the $J$-invariant of orthogonal groups, and, consequently, automatically gives new information about the structure of the Morava motives of maximal orthogonal Grassmannians (see Corollary~\ref{cormax}).

For example, as an immediate corollary of the above theorem and \cite[Theorem~5.7]{PShopf}
we obtain that for $m>2^{n+1}$ the $\mathrm K(n)$-motive of a connected component of the maximal orthogonal Grassmannian for a generic $m$-dimensional quadratic form with trivial discriminant is always non-trivially decomposable (see Corollary~\ref{cormax}). Note that contrary to this by a result of Vishik its Chow motive is always indecomposable.

We also provide a complete computation of decompositions of the Morava motives of generic quadrics. 
Note that one can abstractly specialize this decomposition to an arbitrary smooth projective quadric, and we explicitly write the corresponding projectors.

The results of Levine--Morel~\cite{LM} and Vishik--Yagita~\cite{VY} imply that the decomposition of the Chow motive is the ``roughest'' one in the sense that any decomposition of the Chow motive 
of a smooth projective variety $X$ can be lifted to a decomposition of the $A^*$-motive of $X$ for every oriented cohomology theory $A^*$.

However, for an indecomposable Chow summand the corresponding $A^*$-motive can be decomposable. For example, the Chow motive of a generic quadric is indecomposable (see~\cite{Vlect,Ksuff}), on the other hand, by results of Swan and Panin \cite{Swan,Pkt} the $\mathrm K^0$-motive of every positive dimensional  quadric is decomposable and encodes, for example, its Clifford invariant.

The application of our computations to the motivic study of quadrics is the following.

\begin{tm}[Theorem~\ref{result}]\label{thm:intro_quad}
Let $Q$ be a generic quadric of dimension $D$ and $n>1$. We denote $N=2^n$ for $D=2d$ even or $N=2^n-1$ for $D=2d+1$ odd. 
\begin{enumerate}[{\rm (i)}]
    \item If $D<2^n-1$, then the $\mathrm K(n)$-motive of $Q$ is indecomposable. 
    \item If $D\geq 2^n-1$, then the $\mathrm K(n)$-motive of $Q$ decomposes as a direct sum of $2d+2-N$ Tate motives $\mathcal M(\mathrm{pt})\{i\}$ and an indecomposable summand of rank~$N$. 
    \end{enumerate}
 Moreover, if $D$ is odd, then over an algebraic closure this indecomposable summand splits as a sum of different twists of the motive of the point $\bigoplus_{i=0}^{2^n-2}\mathcal M(\mathrm{pt})\{i\}$, and if $D$ is even, then one of these twists appears twice.
\end{tm}

Note that the last part of this theorem is important due to the periodicity of Tate motives for the $n$-th Morava K-theory with period $2^n-1$. Namely, one has $\mathcal M(\mathrm{pt})\cong\mathcal M(\mathrm{pt})\{2^n-1\}$.

The above theorem is a nice illustration of a general principle that the Morava K-theories provide a sequence of deformations from $\mathrm K^0$ to the Chow theory. For example, in the odd-dimensional case the rank of the indecomposable summand in the statement of the above theorem is (when $n$ is growing) $3,7,15,\ldots$ and eventually stabilizes at $2^{\lfloor\log_2(D+1)\rfloor}-1, D+1, D+1, \ldots$ when the motive of the whole quadric becomes indecomposable as it should also be according to results of Karpenko and Vishik for Chow motives.

The proof of the above theorem uses techniques of \cite{PShopf} and, in particular, relies on the explicit computations of the $A^*(\mathrm{SO}_m)$-comodule structure of $A^*(\overline Q)$ where $m=\mathrm{dim}\,\overline Q+2$, $A^*=\mathrm K(n)^*$, and $\overline Q$ is a split quadric.

\subsection{Acknowledgement}
We would like to thank Alexey Ananyevskiy for his useful suggestions and the anonymous referee for a careful reading of the preliminary version of this article and for his valuable comments.

The first-named author was supported by the DFG project AN 1545 ``Equivariant and weak orientations in the motivic homotopy theory''. The first-named and the second-named authors of the article were supported by the SPP 1786 ``Homotopy theory and algebraic geometry'' (DFG). The third-named author was supported by the BASIS foundation grant ``Young Russia Mathematics''. A part of this work was written when the second-named author was in St. Petersburg University supported by the BASIS foundation grants ``Young Russia Mathematics'' and the Ministry of Science and Higher Education of the Russian Federation, agreement No. 075–15–2022–287.

The results of Section 4--5, 7 constitute the major part of the second author's Ph. D. Thesis~\cite[Chapter~I,\,II]{LavPhD}, however, the proofs in this paper are simplified and the exposition is clarified.

\section{Background on oriented cohomology theories}

\subsection{Oriented cohomology theories}
\label{non-grad}

In the present article we always work over a field $k$ of characteristic $0$ and  $\Sm$ denotes the category of smooth quasi-projective varieties over $k$. We usually denote $\mathrm{Spec}\,k$ by $\mathrm{pt}$. 

Consider an oriented cohomology theory $A^*$ in the sense of Levine--Morel~\cite{LM} over $k$. Each oriented cohomology theory is equipped with a (commutative and one-dimensional) formal group law $F_A$, e.g., for $A^*=\mathrm{CH}^*$ the corresponding formal group law is additive, and for Grothendieck's  $\mathrm K^0[\beta^{\pm1}]$ the corresponding formal group law $F_{\mathrm{K}^0}(x,y)=x+y-\beta xy$ is multiplicative.

Levine and Morel constructed a theory of {\it algebraic cobordism} $\Omega^*$, which is a universal oriented cohomology theory. They also proved that $\Omega^*(\mathrm{pt})$ is isomorphic to the Lazard ring~$\mathbb L$, and the formal group law $F_{\Omega}$ is the universal formal group law~\cite[Theorem~1.2.7]{LM}. As a consequence, for any commutative $\mathbb Z$-graded ring $R$, and any formal group law $F(x,y)$ homogeneous of degree $1$ as an element of $R[[x,y]]$, there exists an oriented cohomology theory $A^*$, more precisely,
$$
A^*=\Omega^*\otimes_{\mathbb L}R
$$
such that $A^*(\mathrm{pt})=R$, and the corresponding formal group law $F_A(x,y)$ is equal to $F(x,y)$. Such a theory $A^*$ is called {\it free}. 

Free theories keep many properties of the algebraic cobordism, for instance, they are {\it generically constant} in the sense of~\cite[Definition~4.4.1]{LM}, see~\cite[Corollary~1.2.11]{LM}, and satisfy the localization property~\cite[Definition~4.4.6, Theorem~1.2.8]{LM}. Free theory $A^*$ also satisfies the following identity, which we call the {\it Normalization identity} following~\cite[Theorem~1.1.8]{Prr}:
\begin{align}
\label{divisor}
\iota_A\big(1_{A^*(D)}\big)=c_1^A\big(\mathcal L(D)\big)
\end{align}
for any smooth divisor $\iota\colon D\hookrightarrow X$, and $\mathcal L(D)$ the corresponding line bundle as in~\cite[Chapter~II, Proposition~6.13]{Har}, see~\cite[Proposition~5.1.11]{LM}. Actually, we will use it only for $D$ a hypersurface in $\mathbb P^n$ with $\mathcal L(D)=\mathcal O(d)$, and $A^*=\Omega^*$ the algebraic cobordism theory.

The examples of free theories are Chow, $\mathrm K^0[\beta^{\pm1}]$, and Morava K-theories~\cite{Log,Se1,Se2,SeSe}.

\subsection{Topological filtration}
\label{filtration}

Assume that an oriented cohomology theory $A^*$ satisfies the localization property~\cite[Definition~4.4.6]{LM}, i.e., 
$$ 
A^*(Z) \xrarr{i_*} A^*(X) \xrarr{j^*} A^*(U)\xrarr{\ } 0
$$
is exact for a closed subvariety 
$i:Z\rarr X$ of a smooth variety $X$ with the open complement $j:U\rarr X$ (see~\cite[Remark~2.1.4]{LM} for the definition of $A^*(Z)$ for $Z$ not smooth). Then the two possible definitions of the topological filtration coincide:
$$
\tau^s A^*(X):=\!\!\!\!\!\bigcup_{\substack{U\text{ open}\\\mathrm{codim}_X X\setminus U \ge s}} \!\!\!\!\!\mathrm{Ker\ } \big(A^*(X)\rarr A^*(U)\big) = \!\!\!\!\bigcup_{\substack{Z\text{ closed}\\\mathrm{codim}_X Z\ge s}} \!\!\!\!\mathrm{Im\ } \big(A^*(Z)\rarr A^*(X)\big).
$$
In fact, we will only consider free theories $A$ in the present article.

We denote $\mathrm{gr}_\tau^*A^*(X)$ the associated graded ring of $A^*(X)$, and we consider the canonical map
$$
\rho\colon A^*(\mathrm{pt})\otimes_{\mathbb Z}\mathrm{CH}^*(X)\twoheadrightarrow\mathrm{gr}_\tau^*A^*(X)
$$
obtained by the universality of the Chow theory in the class of theories with the additive formal group law~\cite[Theorem~4.5.1]{LM}. The map $\rho$ sends the class $a\otimes[Z]$ for $a\in A^{-s}(\pt)$, $i\colon Z\hookrightarrow X$ a closed integral subscheme of codimension $p$, to the class of $\,a^{\,}\cdot^{\,}(i\circ f)_A\big(\widetilde Z\big)$ in $\mathrm{gr}_\tau^pA^{p-s}(X)$, where $f\colon\widetilde Z\rightarrow Z$ is a projective birational morphism with $\widetilde Z$ smooth, cf.~\cite[Lemma~4.5.3]{LM} and~\cite[Corollary~4.5.8]{LM}. The map $\rho$ is surjective for any free theory $A^*$ (in fact, for much more general $A^*$, see~\cite[Theorem~4.4.7]{LM}).

If $A^*$ is a free theory such that $A^{>0}(\pt)=0$, then the topological filtration on $A^*(X)$ is defined purely in terms of the structure of an $A^*(\pt)$-graded module. Namely, 
\begin{align}
\label{topfilt}
\tau^s A^n(X) = \sum_{-m\le n-s} A^{-m}(\pt)\cdot A^{n+m}(X)
\end{align}
(see~\cite[Theorem~4.5.7]{LM} and \cite[Section~4]{Vcob} for the case of algebraic cobordism). In particular, $\tau^sA^*(X)$ coincides with the ideal of $A^*(X)$ generated by all homogeneous elements of codimension at least~$s$.

\begin{rk*}
Observe that our choice of gradings on the associated graded ring coincides with the one in~\cite[Section~4]{Vcob}, but Levine and Morel in~\cite[Subsection~4.5.2]{LM} denote by $\mathrm{Gr}_sA^{p-s}(X)$ what we denote by $\mathrm{gr}_\tau^pA^{p-s}(X)$.
\end{rk*}

\subsection{Morava K-theory}
In the article we use notions from the theory of formal group laws (see \cite{Rav,Haz}).
For a prime $p$ consider the universal $p$-typical formal group law $F_{\mathrm{BP}}$ defined over the ring $V\cong\mathbb Z_{(p)}[v_1,\,v_2,\ldots]$, see~\cite[Theorem~A2.1.25]{Rav} (here $\mathbb Z_{(p)}$ denotes the localization of $\mathbb Z$ at the ideal $(p)=p\mathbb Z$). If ${\mathrm{log}_{\mathrm{BP}}(t)=\sum_{i\geq0}l_it^{p^i}}$ is the logarithm of $F_{\mathrm{BP}}$ over $V\otimes\mathbb Q$, then the mentioned isomorphism ${V\cong\mathbb Z_{(p)}[v_1,\,v_2,\ldots]}$ can be chosen in such a way that $p\,^{\!}l_k=\sum_{i=0}^{k-1}l_iv_{k-i}^{p^i}$, see~\cite[Theorem~A2.2.3]{Rav} (in this situation $v_k$ are often called {\it Hazewinkel's generators}). Then choosing any $n$, and sending $v_k$ to $0$ for $k\neq n$ we obtain a formal group law with the logarithm
\begin{align}
\label{loggr}
l(t)=\sum_{k\geq0}\,p^{-k}\,v_n^{\frac{p^{nk}-1}{p^n-1}}\,t^{p^{nk}}\in\mathbb Q[v_n][[t]]
\end{align}
where $v_n$ has degree $1-p^n$. We call the corresponding free theory
$$
\mathrm K(n)^*\big(-;\,\mathbb Z_{(p)}[v_n^{\pm1}]\big)=\Omega^*\otimes_{\mathbb L}\mathbb Z_{(p)}[v_n^{\pm1}]
$$ 
the ($n$-th) Morava K-theory. 

The term ``Morava K-theory'' can denote a family of free theories, as in~\cite{Se1,Se2,SeSe}. In the present work we prefer to use it, in contrast, only for the theory chosen above as in~\cite{Log,PShopf}.

It is sometimes also convenient to consider the corresponding {\it connective} version of the Morava K-theory defined as a free theory
$$
\mathrm{CK}(n)^*\big(-;\,\mathbb Z_{(p)}[v_n]\big)=\Omega^*\otimes_{\mathbb L}\mathbb Z_{(p)}[v_n]
$$ 
with the same formal group law and another coefficient ring, or a version with coefficients in a $\mathbb Z_{(p)}$-algebra $R$, 
$$
\mathrm{K}(n)^*\big(-;\,R[v_n^{\pm1}]\big)=\Omega^*\otimes_{\mathbb L}R[v_n^{\pm1}]=\mathrm{K}(n)^*\big(-;\,\mathbb Z_{(p)}[v_n^{\pm1}]\big)\otimes R,
$$
e.g., for $R=\mathbb Z_p$, $\mathbb F_p$ or $\mathbb Q$. In the present article we work only with $\mathrm{K}(n)^*\big(-;\,\mathbb F_2[v_n^{\pm1}]\big)$, and, therefore, use the notation $\mathrm K(n)^*(-)$ for it for shortness. We remark that in the present article we usually consider Morava K-theo\-ries for the prime $p=2$ because we are mainly interested in quadrics.

We remark finally that the algebraic Morava K-theory as conjectured by Voevodsky in~\cite{VoevMor} or as constructed in~\cite{LT} is a {\it bi}-graded ``big'' theory, and in the present article we only consider oriented cohomology theories in the sense of~\cite{LM}, called sometimes ``small''. Our (small) Morava K-theory is the $(2*,*)$-diagonal of the ``big'' theory of~\cite{LT}, as shown in~\cite{Ldiag}.

\subsection{Cellular varieties}

For an oriented cohomology theory $A^*$ we consider the category of (effective) $A$-motives defined as in~\cite{Manin}. In particular, the morphisms between two smooth projective irreducible varieties in this category are given by $A^{\dim Y}(X\times Y)$. The motive of a smooth projective variety $X$ is denoted by $\mathcal M_A(X)=\mathcal M(X)$, and the Tate twist of $\mathcal M$ is denoted by $\mathcal M\{n\}$. We call twisted motives of the point the {\it Tate motives}, and we say that the motive $\mathcal M(X)$ of a smooth projective variety $X$ is {\it split}, if it is isomorphic to a direct sum of Tate motives.

We call a smooth projective variety $X\!$ {\it cellular} if there exists a filtration 
\begin{align}
\label{cellular}
X=X_0\supseteq X_1\supseteq\ldots\supseteq X_{n+1}=\emptyset
\end{align}
of $X$ by closed subschemes such that $X_i\setminus X_{i+1}$ is a disjoint union of affine spaces for each $i$. By~\cite[Corollary~2.9]{VY} (and by the universality of the algebraic cobordism) the $A$-motive $\mathcal M_A(X)$ of a cellular variety $X$ is split.

For $X$ cellular and for any smooth projective $Z$, we can obtain from~(\ref{cellular}) a similar filtration for $X\times Z$, and deduce by the same kind of arguments that the motive of $X\times Z$ decomposes to a sum of twisted motives of $Z$, cf.~\cite{NZ}. This implies, in particular, that the map $$A^*(X)\otimes_{A^*(\mathrm{pt})}A^*(Z)\rightarrow A^{*}(X\times Z)$$ given by $x\otimes z\mapsto x\times z$ is an isomorphism. We refer to this fact as the {\it K\"unneth formula}.

Natural examples of cellular varieties are {\it split} pro\-jec\-tive homogeneous varieties.

\subsection{The Rost nilpotence principle}
\label{rostnilp}

Let $L/k$ be a field extension, and consider the {\it extension of scalars} map $\mathrm{res}^A_{L/k}$ for a free theory $A^*$, see~\cite[Example~1.2.10]{LM} and~\cite[Example~2.7]{GV}. The map $\mathrm{res}^A_{L/k}$ can be extended to the category of $A$-motives, and for $\mathcal M\in\AMot$ we write $\mathcal M_L$ for $\mathrm{res}^A_{L/k}\,\mathcal M$. We say that the {\it Rost nilpotence principle} holds for $\mathcal M$ if for every field extension $L/k$ the kernel of the map
\begin{align*}
\mathrm{End}_{\Mot A}(\mathcal M)\rightarrow\mathrm{End}_{\Mot A}(\mathcal M_L).
\end{align*}
consists of nilpotent elements.

The Rost nilpotence principle is proven for the Chow motives of projective quadrics by Rost in~\cite{Rost}. Alternative proofs of this result are given in~\cite{Vrost} and~\cite{Brost}. The Rost nilpotence principle for Chow motives of projective homogeneous varieties is proven in~\cite{CGM} and~\cite{Bdec}, for generically split Chow motives of smooth projective varieties in~\cite{VZrost}, and for Chow motives of surfaces in~\cite{Grost,Grost2}. By~\cite[Corollary~2.8]{VY} the Rost nilpotence principle holds in these cases for algebraic cobordism or connective K-theory motives as well.

A recent result of Gille--Vishik~\cite[Corollary~4.5]{GV} asserts that the Rost nilpotence principle holds for the motive $\mathcal M(X)$ of every projective homo\-geneous variety  $X$ for a semisimple algebraic group (e.g., for the motive $\mathcal M(Q)$ of any smooth projective quadric $Q$) and for every free theory $A^*$.

\subsection{The Krull--Schmidt Theorem}
\label{Krull-Schmidt}

We say that the Krull--Schmidt Theorem holds for a motive $\mathcal M$ if for every two decompositions of $\mathcal M$ into finite direct sums of {\it indecomposable} motives 
$$
\mathcal M\cong\bigoplus_{i=1}^m\mathcal N_i\cong\bigoplus_{j=1}^{m'}\mathcal N'_j
$$
one has $m=m'$ and $\mathcal N'_j\cong\mathcal N_{\sigma(i)}$ for some permutation $\sigma\in S_m$.

Recall that a (non-commutative) ring $S$ is {\it local} if for all $a,b\in S$ with $a+b=1$ it follows that at least one of $a$, $b$ is invertible. The following proposition is a useful tool to prove the Krull--Schmidt Theorem~\cite[Chapter~I, Theorem~3.6]{Bass}.

\begin{prop}
\label{bass}
Assume that an additive category $\mathcal C$ is idempotent complete. Let $A_i$, $B_j$ be objects of $\mathcal C$ with local endomorphism rings and
$$
A_1\oplus\ldots\oplus A_m=B_1\oplus\ldots\oplus B_{m'}.
$$
Then $m=m'$ and $A_i\cong B_{\sigma(i)}$ for some $\sigma\in S_m$.
\end{prop}

In particular, if any indecomposable direct summand of a motive $\mathcal M$ has a local endo\-morphism ring, we conclude that the Krull--Schmidt Theorem holds for $\mathcal M$. This is the case, e.g., when the theory $A^*$ is free and has the property that $A^*(\mathrm{pt})$ is a $K$-algebra over some field $K$, all $A^k(\mathrm{pt})$ are finite dimensional vector spaces over $K$ (e.g., for $A^*=\mathrm K(n)^*$), and $\mathcal M=\mathcal M(X)$ for a smooth projective homogeneous variety $X$. Indeed, since $\mathcal M(X_{\,\overline k})$ is split, for an indecomposable (non-zero) summand $\mathcal N$ of $\mathcal M(X)$ the ring $S=\mathrm{res}_{\,\overline k/k}\,\mathrm{End}({\mathcal N})$ is a finite-dimensional algebra over $K$. Since the Rost nilpotence principle holds for $\mathcal M(X)$, we conclude that $S\neq 0$, and $S$ is indecomposable by~\cite[Chapter~III, Proposition~2.10]{Bass}. This implies that $S$ is local~\cite[Corollary~19.19]{Lfc}. Finally, apply the Rost nilpotence principle again to conclude that $\mathrm{End}({\mathcal N})$ is local itself.

Consider a twisted form $X$ of a cellular variety, i.e., such that $X_L$ is cellular for a field extension $L/k$, and consider a free theory $A^*$ such that the Rost nilpotence principle and the Krull--Schmidt Theorem hold for the $A$-motive of $X$. Assume additionally that the $A$-motive of a point is indecomposable. Then for every summand $\mathcal N$ of $\mathcal M_A(X)$ its image under the restriction map $\mathrm{res}_{L/k}$ is a sum of Tate motives. The number of these Tate motives is called the {\it rank} of $\mathcal N$. For $\mathcal N\neq 0$ the Rost nilpotence principle implies that the rank of $\mathcal N$ is greater than $0$. Clearly, using induction on the rank, one can show that the motive of $X$ admits a decomposition into a {\it finite} direct sum of indecomposable summands. 

We remark that the Krull--Schmidt Theorem for $A^*=\mathrm{CH}^*$ with integral coefficients is proven for motives of quadrics in~\cite{Vlect}, cf. also~\cite{Hau}. However, the counterexamples~\cite[Example~9.4]{ChM} and \cite[Corollary~2.7]{CPSZ} provide (twisted) projective homogeneous varieties for which the Krull--Schmidt Theorem fails (see also~\cite{SeZh}). In~\cite{ChM} Chernousov and Merkurjev proved the Krull--Schmidt Theorem for motives of projective varieties homogeneous under a simple group or an inner form of a split reductive group and $A^*=\mathrm{CH}^*(-;\,\mathbb Z_{(p)})$.

\section{Background on algebraic groups}

\subsection{Schubert calculus}
\label{schcal}

Let $G$ denote a split semisimple group of rank $l$ over $k$, $T\cong(\mathbb G_{\mathrm{m}})^{\times\,l}$ a fixed split maximal torus of $G$, and $B$ a Borel subgroup containing $T$. Let $\Pi=\{\alpha_1,\ldots,\alpha_l\}$ be the respective set of simple roots of the root system $\Phi$ of $G$. 
For a subset $\Theta\subseteq\Pi$ consider the corresponding parabolic subgroup $P_\Theta$ in $G$, generated by $B$ and by the root subgroups $U_{-\alpha}$ for $\alpha\in\Theta$. In particular, $B=P_\emptyset$, and the maximal parabolic subgroups are $P_i=P_{\Pi\setminus\{\alpha_i\}}$.

For a parabolic subgroup $P$ in $G$ we denote its Levi part by $L$, its unipotent radical by $U$, and the opposite unipotent radical by $U^-$. We also denote the commutator subgroup of $L$ by $C=[L,\,L]$. In the particular case $G=\mathrm{SO}_m$, $P=P_1$, and $L=L_1$ its Levi subgroup, the group $C$ is isomorphic to $\mathrm{SO}_{m-2}$.

The Weyl group $W$ of $G$ is generated by simple reflections $s_i=s_{\alpha_i}$ corresponding to simple roots $\alpha_i\in\Pi$. We denote by $l(v)$ the length of $v\in W$ in simple reflections. 
The longest word of $W$ is denoted by $w_0$. 

For $P=P_\Theta$ consider ${W_P=\langle s_i\mid i\in\Theta\rangle}$ and let us denote $$W^P=\{v\in W\mid l(vs_i)=l(v)+1\ \ \forall\, i\in\Theta\}.$$ Then the map
$$
W^P\times W_P\rightarrow W
$$
sending a pair $(u,v)$ to the product $uv$ is a bijection and $l(uv)=l(u)+l(v)$~\cite[Proposition~2.2.4]{BBcox}. This immediately implies that $W^P$ is a set of the minimal representatives for the elements of $W/W_P$.

Consider now the case $A^*=\mathrm{CH}^*$, and let $X_w$ denote the classes $[\overline{BwB/B}]$ of Schubert varieties in $\mathrm{CH}^*(G/B)$. Observe that these varieties are not necessarily smooth. It is well-known that $\{X_w\mid w\in W\}$ is a free basis of $\mathrm{CH}^*(G/B)$. 
It can be also convenient to replace $B$ by the corresponding opposite parabolic subgroup $B^-$ in the Bruhat decomposition, and introduce notation $Z_w$ for the classes of $\overline{B^-wB/B}$ in $\mathrm{CH}(G/B)$. Observe that $G$ acts on $G/B$ with left translations in such a way that $w_0(BwB/B)=B^-w_0wB/B$ as subvarieties, and taking closures we obtain
$$
w_0\,\big(\overline{BwB/B}\big)=\overline{B^-w_0wB/B}.
$$
By~\cite[Lemma~1]{Gr} the induced action of $G$ on $\mathrm{CH}(G/B)$ is trivial, and, therefore, $X_w=Z_{w_0w}$ (see also~\cite[Proposition~1]{Dem}). 

More generally, the classes $X_w=[\overline{BwP/P}]$ for $w\in W^P$ form a free basis of $\CH^*(G/P)$. In the dual notation we set $Z_w=X_{w_0ww^P}\in\CH^*(G/P)$, where $w^P$ denotes the longest element of $W_P$. Then the pullback map along the natural projection $\pi\colon G/B\rightarrow G/P$ sends $Z_w\in\CH^*(G/P)$ to $Z_w\in\CH^*(G/B)$ for all $w\in W^P$ identifying $\mathrm{CH}^*(G/P)$ with the free abelian subgroup of $\mathrm{CH}^*(G/B)$ generated by $\{Z_{w}\mid w\in W^{P}\}$, see also~\cite[Section~5.1]{GPSshells}.

The next theorem is proven in~\cite[Theorem~13.13]{CPZ}.

\begin{tm*}[Calm\`es--Petrov--Zainoulline]
\label{yfree}
For any free theory $A^*$ there exists a free $A^*(\mathrm{pt})$-basis of $A^*(G/B)$ consisting of homogeneous elements $\zeta_w=\zeta_w^A$, ${w\in W}$. Moreover, $\zeta^A_w$ coincide with the images of $\zeta^\Omega_w$ under the canonical map ${\Omega^*(G/B)\rightarrow A^*(G/B)}$ and $\zeta^{\mathrm{CH}}_w$ coincide with $X_w$ defined above.
\end{tm*}

More precisely, $\zeta_w$ are given by resolutions of singularities of Schubert varieties $\overline{BwB/B}$.

\subsection{The characteristic map}
\label{characteristic}

Recall that $A^*(\mathrm BT)$ for a classifying space $\mathrm BT$ of a split torus $T$ is isomorphic to $A^*(\mathrm{pt})[\![x_1,\ldots,x_l]\!]$. If $M$ is a group of characters of $T$ we can identify $A^*(\mathrm BT)$ with $A^*(\mathrm{pt})[\![M]\!]_{F_A}$ as defined in~\cite[Definition~2.4]{CPZ}, see~\cite[Theorem~3.3]{CZZ}, where $F_A$ is the formal group law of the theory $A^*$. 
We fix a basis $\chi_1,\ldots,\chi_{l}$ of $M$ and write $x_{\chi_i}$ for $x_i\in A^*(\mathrm BT)$. For an arbitrary $\lambda\in M$ we have an element $x_\lambda\in A^*(\mathrm BT)$ defined according to the rule  $x_{\lambda+\mu}=F_A(x_\lambda,\,x_\mu)$.

Following~\cite[Remark~1.2.12]{LM} for any free theory $A^*$ and any smooth irreducible variety $X$ with the function field $K$ we consider a map $\mathrm{deg}_A$ defined as the restriction to the generic $K$-point of $X_K$
$$
\mathrm{deg}_A\colon A^*(X)\rightarrow A^*(K)\cong A^*(\mathrm{pt}).
$$
We remark that the map $\mathrm{deg}_A$ canonically splits by 
$$
A^*(\mathrm{pt})\rightarrow A^*(X),\ \ a\mapsto a\cdot 1_{A^*(X)}.
$$
For an augmented algebra $A$ we denote $A^+$ its augmentation ideal and we say that the sequence of augmented algebras $$(A_i,d_i\colon A_i\rightarrow A_{i+1})$$ is exact if $\mathrm{Ker}\,d_i$ coincides with the ideal generated by $\mathrm{Im}\,d_{i-1}\cap A_i^+$. 

By~\cite[Proposition~5.1, Example~5.6]{GZ} the sequence
\begin{align}
\label{Kr12}
\Omega^*(\mathrm BT)\rightarrow\Omega^*(G/B)\rightarrow\Omega^*(G)\rightarrow\mathbb L
\end{align}
is a right exact sequence of augmented $\mathbb L$-algebras, where the first arrow is the {\it characteristic map} $\mathfrak c=\mathfrak c_\Omega$ as defined in~\cite[Definition~10.2]{CPZ} and the second one is the pullback along $G\rightarrow G/B$. This, obviously, gives a similar sequence for any free theory $A^*$.
Observe that the above sequence can be continued to the left as in~\cite[Theorem~10.2]{CZZ}.

\subsection{BGG--Demazure operators}
\label{BGGD}

Divided difference operators $\Delta_i$ were defined independently by Bernstein--Gelfand--Gelfand~\cite{BGG} and Demazure~\cite{Dem} to describe the characteristic map $\mathfrak c_{\mathrm{CH}}$. These operators were generalized to arbitrary oriented cohomology theories in~\cite{CPZ}.

Observe that the action of the Weyl group $W$ on the group $M$ of characters of $T$ induces an action of $W$ on $A^*(\mathrm BT)=A^*(\mathrm{pt})[\![M]\!]_{F_A}$ according to the rule $s_\alpha(x_\lambda)=x_{\lambda-\alpha^\vee(\lambda)\alpha}$. We now define divided difference operators for $\Omega^*$ by the formula
$$
\Delta^\Omega_i(u)=\frac{u-s_i(u)}{x_{\alpha_i}},
$$
and for any other theory by the change of coefficients~\cite[Definition~3.5]{CPZ}.

Consider the case $A^*=\mathrm{CH}^*$. Then for any two reduced decompositions $w=s_{i_1}\ldots s_{i_k}=s_{j_1}\ldots s_{j_k}$ the operators $\Delta_{i_1}\circ\ldots\circ\Delta_{i_k}$ and $\Delta_{j_1}\circ\ldots\circ\Delta_{j_k}$ coincide and we denote such a composition simply by $\Delta_w$. Moreover, if a decomposition $s_{i_1}\ldots s_{i_k}$ is not reduced, then $\Delta_{i_1}\circ\ldots\circ\Delta_{i_k}=0$, see~\cite[Theorem~3.4]{BGG}.

We have the following description of $\mathfrak c_{\mathrm{CH}}$ in terms of divided difference operators. For a homogeneous $u\in\mathrm{CH}^*(\mathrm BT)$ of degree $s$ one has by~\cite[4.5.~Th\'eor\`eme~1\,(b)]{Dem}
$$
\mathfrak c_{\mathrm{CH}}(u)=(-1)^{l(w_0)-s}\sum_{l(w)=s}\Delta_w(u)Z_w,
$$
where $\Delta_w(u)\in\mathrm{CH}^0(\mathrm BT)=\mathbb Z$ since $l(w)=s$.

Further, for a minimal parabolic subgroup $P_{\{\alpha_i\}}$ consider the natural projection 
$$
\pi_i\colon G/B\rightarrow G/P_{\{\alpha_i\}},
$$
and define an operator $\widetilde{\Delta}_i(z)=-(\pi_i^{\mathrm{CH}}\circ(\pi_i)_{\mathrm{CH}})(z)$ on $\mathrm{CH}^*(G/B)$. Then for the characteristic map $\mathfrak c_{\mathrm{CH}}$ one has $\widetilde{\Delta}_i\circ\mathfrak c_{\mathrm{CH}}=\mathfrak c_{\mathrm{CH}}\circ\Delta_i$~\cite[Theorem~13.13]{CPZ}. 

It also follows, e.g., from~\cite[Lemma~13.3]{CPZ}, that the sequence $\widetilde{\Delta}_{i_1}\circ\ldots\circ\widetilde{\Delta}_{i_k}$ does not depend on a reduced decomposition of $w=s_{i_1}\ldots s_{i_k}$ and equals $0$ if it is not reduced. We write $\Delta_i$ for $\widetilde{\Delta}_i$ if it does not lead to a confusion.

We can also describe the action of ${\Delta}_i$ on the Schubert basis. Take $w\in W$ and assume that $l(ws_i)=l(w)+1$ for some $i$. Then by~\cite[Lemma~13.3]{CPZ} and~\cite[Lemma~12.4]{CPZ} we have $\Delta_{i}(X_{w})=-X_{ws_i}$.

This means in particular that every $X_w$ can be obtained from $\pm\mathrm{pt}=\pm X_{1}$ by a sequence of $\Delta_i$ and, moreover, since $w_0$ is greater than $w$ in the weak Bruhat order for every $w\in W\setminus\{w_0\}$~\cite[Definition~3.1.1\,{\rm(}i{\rm)}]{BBcox}, there always exists a sequence $(i_1,\ldots,i_k)$ such that 
\begin{align}
\label{bgg-schubert-duality}
\Delta_{i_1}\circ\ldots\circ\Delta_{i_k}(X_w)=\pm X_{w_0}=\pm1\in\mathrm{CH}^*(G/B)
\end{align}
(here $s_{i_1}\ldots s_{i_k}$ is a reduced expression for $w^{-1}w_0$).

Moreover, if we apply the same operator $\Delta_{w^{-1}w_0}$ to another 
$$w'=s_{j_1}\ldots s_{j_s}\neq w
$$ 
with $s=l(w')=l(w)$, we obtain $\Delta_{w^{-1}w_0}(X_{w'})=0$, since $s_{i_1}\ldots s_{i_k}\cdot s_{j_1}\ldots s_{j_s}$ is not a reduced expression. In this sense $\Delta_w$ and $X_w$ are dual to each other.

\begin{rk*}
The proof of~\cite[Theorem~13.13]{CPZ} contains a misprint, more precisely, $A_{I_w}(z_0)$ should be changed to $A_{I_w^{\mathrm{rev}}}(z_0)$.
\end{rk*}

Similar results are obtained in~\cite{CPZ} for any theory $A^*$. We introduce the notation 
$$
F_A(x,\,y)=x+y+xy\cdot G(x,y),
$$ 
and for any simple root $\alpha_i\in\Pi\subset M$ we denote 
$$
\kappa_i=\kappa_i^A=G\big(\mathfrak c_A(x_{\alpha_i}),\,\mathfrak c_A(x_{\alpha_{-i}})\big)
$$
(observe that $\kappa_i^{\mathrm{CH}}=0$). As above, for a minimal parabolic subgroup $P_{\{\alpha_i\}}$ consider the natural projection $\pi_i\colon G/B\rightarrow G/P_{\{\alpha_i\}}$ and define the Demazure operator $\Delta_i=\Delta_i^A$ on $A^*(G/B)$ by
\begin{align}
\label{sameBGGD}
\Delta_i(z)=\kappa_i\,z-(\pi_i^A\circ(\pi_i)_A)(z),
\end{align}
and the action of simple reflections on $A^*(G/B)$ by
$$
s_i(z)=z-\mathfrak c_A(x_{\alpha_i})\cdot \Delta_i(z).
$$
By~\cite[Lemma~3.8]{PShopf} we have the following Leibniz rule for Demazure operators:
\begin{align}
\label{leibniz}
\Delta_i(uv)=\Delta_i(u)v+s_i(u)\Delta_i(v),
\end{align}
and $s_i$ are $A^*(\mathrm{pt})$-algebra homomorphisms defining an action of the Weyl group $W$ on $A^*(G/B)$. Moreover, $\mathfrak c_A$ respects the action of the Weyl group (see the proof of~\cite[Lemma~3.8]{PShopf}).

\subsection{Torsors}
\label{torsor}

Let $G$ be a smooth affine algebraic group over the base field $k$ of characteristic $0$. We identify the Galois cohomology group $$\mathrm H^1(k,G)=\mathrm H^1\big(\mathrm{Gal}({\overline k}/k),\,G(\overline k)\big)$$ with the set of isomorphism classes of $G$-torsors over $k$. The natural homomorphism $G\rightarrow\mathrm{Aut}(G)$ induces a map on Galois cohomology and, therefore, each $G$-torsor $E$ defines an element $\xi\in\mathrm H^1\big(k,\mathrm{Aut}(G)\big)$ and the corresponding twisted form $\!\,_EG=\!\,_\xi G$ of $G$, see~\cite{Se65} or~\cite[Chapter~VII]{KMRT}.

For a representation $\rho\colon G\hookrightarrow\mathrm{GL}_m$ of $G$ we call $\mathrm{GL}_m\rightarrow\mathrm{GL}_m/\rho(G)$ a {\it standard classifying $G$-torsor}, and its fiber over the generic point of $X=\mathrm{GL}_m/\rho(G)$ we call a {\it standard generic $G$-torsor} over $k(X)$ or just a {\it generic torsor} (see~\cite[Chapter~1, \S5]{GSM}).

Consider a (split) group $G=\mathrm{SO}_m$ over $k$. Let $\!\,_EG$ be a generic group, i.e., the twisted form of $G$ corresponding to a generic $G$-torsor $E$, see~\cite[Section~3]{PSrost}. Strictly speaking, $\!\,_EG$ is defined over a field extension $L$ of $k$, but we will denote $\mathrm{SO}_m$ over $L$ or $\overline L$ by the same letter $G$. For each parabolic subgroup $P$ of $G$ there exists a unique $\!\,_EG$-homogeneous variety $X$ which is a twisted form of $G/P$, see~\cite[Section~1.1]{ChM}, moreover, $X=E/P$. For a maximal parabolic subgroup $P=P_1$ corresponding to the first simple root in the Dynkin diagram (according to the enumeration of Bourbaki) we refer to $E/P_1$ as the {\it generic quadric} of dimension $D=m-2$. We also remark that $E/P_1$ is generic with respect to $\mathrm{SO}_m$, and there exist slightly different notions of generic quadric, which are compared in~\cite{Kgen}.

Until the end of the section we reserve the notation $Q$ for the split quadric $G/P_1$.

For any $G$-torsor $E$ and any free theory $A^*$ there exists a bi-algebra $H^*$ of~\cite[Definition~4.6]{PShopf} defined as an augmented algebra quotient of $A^*(G)$ by the image of $A^*(E)$ and a co-action of $H^*$ on $A^*(G/P)$ for any parabolic subgroup $P$ of $G$~\cite[Definition~4.10]{PShopf}. It is shown in~\cite[Remark~4.8]{PShopf} that for $A^*=\mathrm{CH}(-;\,\mathbb F_p)$ the bi-algebra $H^*$ carries essentially the same information as the {\it J-invariant}~\cite{PSjinv,VJinv}. We will be interested only in the case of a generic torsor $E$. In this case $H^*=A^*(G)$ for every free theory $A^*$~\cite[Example~4.7]{PShopf} and the co-action 
$$
\rho=\rho_A\colon A^*(G/P)\rightarrow A^*(G)\otimes_{A^*(\mathrm{pt})}A^*(G/P)
$$
is given by the pullback map along the left multiplication by elements of $G$ on $G/P$ composed with the K\"unneth isomorphism~\cite[Lemma~4.3]{PShopf}.

Consider smooth projective homogeneous varieties $G/P$, $G/P'\in\Sm$ and let $X$ and $X'$ be the respective $E$-twisted forms. For an element $\alpha\in A^*(X\times X')$ define the {\it realization map}
$$
\alpha_\star\colon A^*(X)\rightarrow A^*(X')
$$
by the formula $(\mathrm{pr}_{X'})_A\circ\mathrm{m}_\alpha\circ\mathrm{pr}_X^A$, where $\mathrm{m}_\alpha$ stands for the multiplication by $\alpha$, and $\mathrm{pr}_X$, $\mathrm{pr}_{X'}$ are the natural projections from $X\times X'$ to $X$ or $X'$ respectively. By~\cite[Theorem~4.14]{PShopf}, 
$$
\overline\alpha_\star\colon A^*(\overline X)\rightarrow A^*(\overline{X'})
$$
is a homomorphism of $A^*(G)$-comodules. This fact allows to consider an (additive) functor from the full subcategory of $\Mot A$ generated by the motives of smooth projective homogeneous $\!\,_EG$-varieties to the category of graded $A^*(G)$-comodules~\cite[Remark~4.15]{PShopf}. In particular, the motivic decomposition of $E/P$ gives us a decompo\-sition of $A^*(G/P)$ into a sum of $A^*(G)$-comodules (as above, with some abuse of notation, we usually write $G/P$ for $X_{\,\overline k}$).

\section{Cobordism ring of a split quadric}\label{multtabl}

\subsection{Multiplication in $\Omega^*(Q)$ of a split quadric $Q$}\label{Qdef}
For any smooth projective cellular variety $X$ and an oriented cohomology theory $A^*$
it is easy to describe the structure of $A^*(X)$ as an abelian group. 
For a split quadric $Q$ we will describe the multiplication in terms of the formal group law. 
For certain theories $A^*$, e.g., for the Morava K-theory $\mathrm K(n)^*$ with $\mathbb F_2$ coefficients, 
the multiplication table has a very simple description. 
We also describe an explicit formula for the pushforward map along the structure morphism from $Q$ to the point.

Let us denote $H=H_{\mathbb P^n}=c_1^\Omega\big(\mathcal O_{\mathbb P^n}(1)\big)$ and recall that the powers of $H^k$ for $0\leqslant k\leqslant n$ form a free basis of $\Omega^*\big(\mathbb P^n\big)$ by the projective bundle formula. Moreover, $H^k$ coincide with the classes of projective subspaces of smaller dimensions in $\mathbb P^n$.

We follow the notation of~\cite[Chapter~XIII, \S\,68]{EKM}. Let $Q$ be a smooth projective quadric of dimension $D$ over $k$ defined by the non-degenerate quadratic form $\varphi$ on a vector space $V$ of dimension $D+2$. We write $D=2d$ for $D$ even or $D=2d+1$ for $D$ odd. We assume that the quadric $Q$ is split, i.e., $\varphi$ has the maximal possible Witt index $d+1$. Let us denote by $W$ a maximal totally isotropic subspace of $V$ and by $W^{\bot}$ its orthogonal complement in $V$ (for $D$ even $W^{\bot}=W$ and for $D$ odd $W$ is a hyperplane in $W^{\bot}$). We denote by $\mathbb P(V)$ the projective space of $V$ which contains $Q$ as a hypersurface. Then $\mathbb P(W)$ is contained in $Q$ and $Q\setminus\mathbb P(W)\rightarrow\mathbb P(V/W^{\bot})$ is a vector bundle
\begin{align}
\label{cell}
\xymatrix{
\mathbb{P}(W)\ar@{^(->}[r]^{\ \ i}&Q&\,Q\setminus\mathbb{P}(W)\ar@{->}[d]_p\ar@{_(->}[l]_{j\ \quad}\\
&&\mathbb{P}(V/W^\bot).
}
\end{align}

We denote 
$$\,H_{\mathbb P(W)}=c_1^\Omega\big(\mathcal O_{\mathbb P(W)}(1)\big)\in\Omega^*(\mathbb P(W)),$$ 
$$
l_i=l_i^\Omega=i_\Omega(H_{\mathbb P(W)}^{d-i}),\ \ \text{and }\ \ h=h_\Omega=c_1^\Omega(\mathcal O_Q(1)).
$$ 

The elements $l_i$ for $0\leqslant i\leqslant d$ and the powers $h^k$ of $h$ for $0\leqslant k\leqslant d$ form a free basis of $\Omega^*(Q)$ over $\mathbb L$, cf., e.g.,~\cite[Co\-rol\-lary~2.9]{VY} or~\cite[Theorem~6.5]{NZ}.

\begin{prop}
\label{table}
In the above notation the multiplication table of $\Omega^*(Q)$ is determined by the identities
\begin{align}
h_{}\cdot l^{}_i&=\begin{cases}l_{i-1}^{},&i>0,\\0,&i=0;\end{cases}\label{hl}\\
l_i^{}\cdot l_j^{}&=\begin{cases}l_0^{},&i=j=d,\,\text{ and }\ D\equiv0\,\ \mathrm{mod}\ 4,\\0,&\text{otherwise};\end{cases}\label{ll}\\
h_{}^{d+1}&=\cfrac{[2]_\Omega(h)}{h^{D-2d}}\,l_d=\sum_{i=1}^{D-d}b_i\,l_{D-d-i}^{}\,\label{hh},
\end{align}
where $[2]_\Omega(t)=F_\Omega(t,t)=\sum_{i\geq1}b_i\,t^i\in\mathbb L[\![t]\!]$ is the multiplication by $2$ in the sense of the universal formal group law.
\end{prop}
\begin{proof}
Observe that the basis elements $l_i$ and $h^k$ are homogeneous elements of $\Omega^*(Q)$ of degree $D-i$ and $k$, respectively, and $\mathbb L$ is graded by non-positive numbers. Therefore, in the decomposition
$$
h^{d+1}=\sum_i a_i\,l_i+\sum_kc_k\,h^k
$$
of $h^{d+1}$ as a sum of basis elements we necessarily have $c_k=0$ for all $k$ by the degree reasons, i.e., $h^{d+1}$ lies in the image of $i_\Omega$. Since $\Omega^*\big(\mathbb P(W))$ injects into $\Omega^*\big(\mathbb P(V)\big)$ we can pushforward the above equality to $\Omega^*\big(\mathbb P(V)\big)$ to determine $a_i$. Let $I\colon Q\hookrightarrow\mathbb P(V)$ denote the inclusion, then
$$
I_\Omega\big(I^\Omega(H^{d+1})\cdot 1\big)=\sum_ia_i\,H^{D+1-i}.
$$
The projection formula implies that the left hand side is equal to $H^{d+1}I_\Omega(1)$ and putting $D=Q$ in~(\ref{divisor}) we have 
$$
I_\Omega(1_{\Omega^*(Q)})=c_1^\Omega\big(\mathcal O(2)\big)=[2]_\Omega\Big(c^\Omega_1\big(\mathcal O(1)\big)\Big)=\sum_{i\geq1}b_iH^i.
$$
Then, clearly, $a_{D-d-i}=b_i$. 
The same argument proves~(\ref{hl}).
 
Finally, for the degree reasons we only need to consider~(\ref{ll}) for $D$ even and $i=j=d$, where $l_d^2=al_0$ for some $a\in\mathbb Z$. Since $\Omega^D(Q)\cong\mathrm{CH}^D(Q)$ by~\cite[Lemma~4.5.10]{LM}, we can determine $a$ modulo $\mathbb L^{<0}$. But for the Chow theory the result is well-known, see, e.g.,~\cite[\S68]{EKM}.
\end{proof}

Our initial computation of the above multiplication table was more awkward and the above simplification is suggested by Alexey Ananyevskiy.

\subsection{Pushforwards along structure morphisms}

For $X\in\Sm$ let $\chi\colon X\rightarrow\mathrm{pt}$ be the structure morphism and let us denote by $[X]\in\mathbb L$ the pushforward of $1\in\Omega^*(X)$ along $\chi$, i.e., $[X]=\chi_\Omega(1_{\Omega^*(X)})$. We will describe $\chi$ for $X=Q$ a split quadric.

\begin{prop}
\label{euler}
In the notation of Proposition~{\ref{table}} we have the following formulae:
\begin{align}
\chi_\Omega(l_i)&=[\mathbb P^i];\label{el}\\
\chi_\Omega(h^k)&=\sum_{j=1}^{D+1-k}\!\!b_j\,[\mathbb P^{D+1-k-j}]\,.\label{eh}
\end{align}
\end{prop}
\begin{proof}
Identity~(\ref{el}) is obvious and~(\ref{eh}) follows from~(\ref{divisor}) with the use of the projection formula
$$
\chi_\Omega(h^k)=\chi_\Omega\big(I_\Omega(I^\Omega H^k)\big)=\chi_\Omega\big(H^k\cdot I_\Omega(1_{\Omega^*(Q)})\big)=\chi_\Omega\left(\sum_{j=1}^{D+1-k}b_j\,H^{k+j}\right).
$$
\end{proof}

Obviously, the results of Propositions~\ref{table} and~\ref{euler} can be applied to any oriented cohomology theory by the universality of algebraic cobordism.

\subsection{Morava with finite coefficients}
\label{morava}

 Let $\mathrm{log}_\Omega(t)$ denote the logarithm of the universal formal group law over $\mathbb L\otimes\mathbb Q$. The identity
\begin{align}
\label{mish}
\mathrm{log}_\Omega(t)=\sum_{i=1}^\infty\cfrac{[\mathbb P^{\,i-1}]}{i}\,t^{i}
\end{align}
is known as the Mishchenko formula~\cite[Theorem~1]{Sh}. 

For an oriented cohomology theory $A^*$ let us denote $[\mathbb P^n]_A=\chi_A(1_{A^*(\mathbb P^n)})$. Then combining~(\ref{mish}) with~(\ref{loggr}) we have
$$
[\mathbb P^{i\,}]_{\mathrm K(n)}=
\begin{cases}
2^{(n-1)k}\,v_n^{\frac{2^{nk}-1}{2^n-1}},&i=2^{nk}-1,\\
0,&i\neq 2^{nk}-1.
\end{cases}
$$

In particular, the assumption $n\geq2$ guaranties that $[\mathbb P^{i\,}]_{\mathrm K(n)}\equiv 0\mod2$ for $i>0$. Moreover,  it is not hard to check that 
$$
[2]_{\mathrm K(n)}(t)\equiv v_nt^{2^n}\!\mod 2
$$
(cf.~\cite[A2.2.4]{Rav}). Thus, Propositions~\ref{table} and~\ref{euler} imply the following

\begin{prop}
\label{mod2}
Consider the free theory $A^*=\mathrm K(n)^*\big(-;\,\mathbb F_2[v_n^{\pm1}]\big)$ with $n\geq1$. Then for a smooth projective split quadric $Q$ of dimension $D=2d+1$ or $D=2d+2$, the ring  $\mathrm K(n)^*\big(Q;\,\mathbb F_2[v_n^{\pm1}]\big)$ is a free $\mathbb F_2[v_n^{\pm1}]$-module with the basis $l^{\mathrm K(n)}_i$ and $(h_{\mathrm K(n)})^k$, $0\leq i,k\leq d$, defined in Proposition~\ref{table}. The multiplication table can be deduced from the identities
\begin{align}
h_{\mathrm K(n)}\cdot l^{\mathrm K(n)}_i&=\begin{cases}l_{i-1}^{\mathrm K(n)},&i>0,\\0,&i=0,\end{cases}\label{hlkn}\\
l_i^{\mathrm K(n)}\cdot l_j^{\mathrm K(n)}&=\begin{cases}l_0^{\mathrm K(n)},&i=j=d,\,\text{ and }\ D\equiv0\,\ \mathrm{mod}\ 4,\\0,&\text{otherwise},\end{cases}\label{llkn}\\
h_{\mathrm K(n)}^{d+1}&=\begin{cases}v_n\,l^{\mathrm K(n)}_{D-d-2^n},&D\geq2^{n+1}-1,\\0,&D<2^{n+1}-1.\end{cases}\label{hhkn}
\end{align}
For $n>1$ the pushforward along the structure morphism is described by
\begin{align}
\chi_{\mathrm K(n)}(l_0^{\mathrm K(n)})&=1,\label{el0kn}\\
\chi_{\mathrm K(n)}(l_i^{\mathrm K(n)})&=0,\ \ i>0,\label{elkn}\\
\chi_{\mathrm K(n)}(h^k_{\mathrm K(n)})&=0,\ \ k\neq D+1-2^n,\label{ehkn}\\
\chi_{\mathrm K(n)}(h^{D+1-2^n}_{\mathrm K(n)})&=v_n,\ \ \text{if}\ D\geq2^n-1.\label{ehD'kn}
\end{align}
\end{prop}

\section{Stabilization of the Morava K-theory for special orthogonal groups}
\label{vitya}

In the present section we prove that $\mathrm K(n)^*\big(\mathrm{SO}_m;\,\mathbb F_2[v_n^{\pm1}]\big)$ stabilizes for $m$ large enough. 
We will use in the proof the multiplication table for the Morava K-theory of a split quadric from Proposition~\ref{mod2} (see Subsection~\ref{bisc-section}).

\subsection{The plan of the proof}
\label{potp}

In the present section and in Sections~\ref{smallso}--\ref{nikita} we only work with the Morava K-theory with coefficients in $\mathbb F_2[v_n^{\pm1}]$ and, therefore, write $$\mathrm K(n)^*=\mathrm K(n)^*\big(-;\,\mathbb F_2[v_n^{\pm1}]\big)$$ for shortness. We are interested in the structure of $\mathrm K(n)^*(\mathrm{SO}_m)$ for a group variety $\mathrm{SO}_m$. 
We will prove the following
\begin{tm}
\label{stab}
For $m\geqslant 2^{n+1}+1$ the pullback map along the natural closed embedding induces an isomorphism of rings
$$
\mathrm K(n)^*(\mathrm{SO}_m)\cong\mathrm K(n)^*(\mathrm{SO}_{m-2}).
$$
\end{tm}

Recall that for a split reductive group $G$ with a split maximal torus $T$ the sequence
\begin{align*}
\mathrm K(n)^*(\mathrm BT)\rightarrow\mathrm K(n)^*(G/B)\rightarrow\mathrm K(n)^*(G)\rightarrow\mathbb F_2[v_n^{\pm1}]
\end{align*}
is a right exact sequence of augmented $\mathbb F_2[v_n^{\pm1}]$-algebras (see Subsection~\ref{characteristic}), where the first arrow is the characteristic map and the second one is the pullback along $G\rightarrow G/B$. 

On the one hand, $\mathrm K(n)^*(G/B)$ is a free $\mathbb F_2[v_n^{\pm1}]$-module with a basis given by resolutions of singularities $\zeta_w$ of Schubert varieties $\overline{BwB/B}$ for $w$ in the Weyl group $W$ of $G$, see Subsection~\ref{schcal}, and we have an explicit description of $\mathrm K(n)^*(\mathrm BT)$ as a ring $\mathbb F_2[v_n^{\pm1}][\![x_1,\ldots,x_l]\!]$. 
On the other hand, the characteristic map $$\mathfrak c_{\mathrm K(n)}\colon\mathrm K(n)^*(\mathrm BT)\rightarrow\mathrm K(n)^*(G/B)$$ has a rather complicated form after these identifications, cf.~Subsection~\ref{BGGD}. 

Instead of working with $\mathfrak c$ directly we will deduce
Theorem~\ref{stab} from Propositions~\ref{bis} and \ref{bisc} below. The proof of Proposition~\ref{bis} will follow the strategy of \cite[Lemma~6.2]{PShopf} for part a) and \cite{PSadd} for part b). The proof of Proposition~\ref{bisc} will be based on the divided difference operators $\Delta_i$ (see Subsection~\ref{BGGD}).

\begin{prop}
\label{bis}
a{\rm)}\,For any free theory $A^*$ and any parabolic subgroup $P$ in a split semisimple group $G$ the pullback map along the closed embedding induces an isomorphism of rings
$$
A^*(P)\cong A^*(G)\otimes_{A^*(G/P)}A^*(\mathrm{pt}).
$$
b{\rm)}\,For the Levi part $L$ of a parabolic subgroup $P$ and $C=[L,\,L]$ the pullback maps along the natural closed embeddings induce isomorphisms of rings
$$
A^*(P)\cong A^*(L)\cong A^*(C).
$$
\end{prop}

\begin{prop}
\label{bisc}
Let $m\geqslant 2^{n+1}+1$, and let $Q$ denote a split projective quadric of dimension $m-2$. Then the natural map 
$$
\mathrm K(n)^*(Q)\rightarrow\mathrm K(n)^*(\mathrm{SO}_m)
$$
factors through $\mathrm K(n)^*(\mathrm{pt})$.
\end{prop} 

The rest of this section is devoted to the proofs of the above propositions and is organized as follows. We start with an analogue of the Leray--Hirsch Theorem for free theories. Next, we use it to prove Proposition~\ref{bis}, and after that we investigate the divided difference operators and prove Proposition~\ref{bisc} and, finally, deduce Theorem~\ref{stab}.

\subsection{The Leray--Hirsch Theorem}

The required isomorphism 
\begin{align}
\label{DoZh}
A^*(P)\cong A^*(G)\otimes_{A^*(G/P)}A^*(\mathrm{pt})
\end{align}
of Proposition~\ref{bis}\,a{\rm)} can be deduced from the $A^*(G/P)$-module isomorphism
$$
A^*(G/B)\cong A^*(G/P)\otimes_{A^*(\mathrm{pt})}A^*(P/B)
$$
obtained in~\cite{DoZh}, see also~\cite{EG,DT}. However, this argument does not prove that~(\ref{DoZh}) is an isomorphism {\it of rings}, not just of $A^*(\mathrm{pt})$-modules. To fix this, we will show that~(\ref{DoZh}) is actually induced by a pullback map. 

In the notation of Proposition~\ref{bis} let us denote by $B'$ the Borel subgroup of $C$, i.e., $B'=B\cap C$. Observe that $P/B$ and $C/B'$ coincide as varieties and the ring $\mathrm{CH}^*(P/B)$ has a free basis consisting of the classes $Z'_w$ of closures of Bruhat cells $(B')^-wB'/B'$ in $C/B'$, or, equivalently, of closures of $(B')^-wB/B$ in $P/B$, $w\in W_P$.  We introduce the following notation.

\begin{nt}[Subring $R$ and subgroup $V$]
\label{rv}
Let us denote by $R$ a free abelian subgroup of $\mathrm{CH}^*(G/B)$ with the basis $Z_w$, $w\in W^P$, and by $V$ a free abelian subgroup with the basis $Z_w$, $w\in W_P$ (cf. Subsection~\ref{schcal}). We have the following isomorphisms of abelian groups:
$$
R\cong\mathrm{CH}^*(G/P),\qquad V\cong\mathrm{CH}^*(P/B),
$$
moreover, the pullback map $\pi^{\mathrm{CH}}\colon\mathrm{CH}^*(G/P)\rightarrow\mathrm{CH}^*(G/B)$ is injective, and $R=\mathrm{Im}(\pi^{\mathrm{CH}})$, i.e., $R$ is a subring of $\mathrm{CH}^*(G/B)$ isomorphic to $\mathrm{CH}^*(G/P)$.
\end{nt}

The next lemma is just~\cite[Proposition~1]{EG} applied to our situation. We remark that $U_P^-\cdot P/B$ is an open subvariety of $G/B$, and the multiplication $$U_P^-\times (P/B)\rightarrow (U_P^-\cdot P)/B$$ defines an isomorphism of schemes.
\begin{lm}
\label{tym}
\label{clemb}
a{\rm)} The map
$
R\,\otimes_{\mathbb Z}V\rightarrow\mathrm{CH}^*(G/B)
$
sending $p\otimes q$ to $pq$ is a bijection.\\
b{\rm)} The isomorphism of abelian groups $\mathrm{CH}^*(P/B)\cong V$ identifying $Z'_w$ with $Z_w$, ${w\in W_P}$, is a section of the pullback map along the closed embedding ${\iota\colon P/B\rightarrow G/B}$.
\end{lm}
\begin{proof}
For a{\rm)}, see~\cite[Lemma~6.2]{PShopf}. 
For b{\rm)}, decompose $\iota$ as a composition of a closed embedding $i\colon P/B\rightarrow (U_P^-\cdot P)/B$ followed by an open embedding $j\colon (U_P^-\cdot P)/B\hookrightarrow G/B$. 
The pullback $j^*(Z_w)$ is given by  $Z_w\cap (U_P^-\cdot P)/B$, and the latter coincides with $U_P^-\times\overline{(B')^-wB/B}$ as a subvariety, since both schemes are reduced. 

Next, $i$ is the zero section of the trivial fibration $U_P^-\cdot P/B\rightarrow P/B$ and, therefore, sends $U_P^-\times\overline{(B')^-wB/B}$ to $Z'_w$.
\end{proof}

Now we will prove an analogue of Lemma~\ref{tym} for any free theory. As a matter of fact, we deduce it from the Chow case with the use of the graded Nakayama Lemma. 

\begin{lm}[Graded Nakayama's Lemma]
\label{surj}
Let $M$ and $N$ be graded $\mathbb L$-modules, let $N$ be finitely generated, and let $f\colon M\rightarrow N$ be a homomorphism of $\mathbb L$-modules which preserves grading. Then\\
a{\rm)} $\mathbb L^{<0}N=N\Rightarrow N=0${\rm;}\\
b{\rm)} if $f\otimes_{\mathbb L}\mathbb Z\colon M/\mathbb L^{<0}M\rightarrow N/\mathbb L^{<0}N$ is surjective, then $f$ is surjective as well.
\end{lm}

\begin{nt}[Subring $R^A$ and free submodule $V^A$]
For an oriented cohomology theory $A^*$ we denote by $R^A$ the image of $A^*(G/P)$ in $A^*(G/B)$ under the pullback map and by $V^A$ the (free graded) $A^*(\mathrm{pt})$-submodule of $A^*(G/B)$ generated by $\zeta^A_{w_0w}$, $w\in W_P$, cf. Subsection~\ref{schcal}.
\end{nt}

\begin{lm}
\label{ck}
a{\rm)} The map
$
f\colon R^A\otimes_{A^*(\mathrm{pt})}V^A\rightarrow A^*(G/B)
$
which sends $p\otimes q$ to  $pq$ is surjective.\\
b{\rm)} The pullback map along the closed embedding
$$
\iota^A\colon A^*(G/B)\rightarrow A^*(P/B)
$$
restricted to $V^A$ defines an isomorphism $V^A\cong A^*(P/B)$.
\end{lm}

\begin{proof}
Since $V^A$ and $A^*(P/B)$ are free $A^*(\mathrm{pt})$-modules of the same rank, it is sufficient to prove that $\iota^A|_{V^A}$ is surjective to get b{\rm)}. 
Therefore, we can assume that $A^*=\Omega^*$. Since the $\mathbb L$-module $M=R^\Omega\otimes_{\mathbb L}V^\Omega$ is graded, $f$ preserves grading, and $f\otimes_{\mathbb L}\mathbb Z$ is surjective by Lemma~\ref{tym}, we can apply Lemma~\ref{surj} to get a{\rm)}. Similarly, we can apply Lemma~\ref{surj} to $M=V^\Omega$, $f=\iota^\Omega|_{V^\Omega}$ to get b{\rm)}.
\end{proof}

\subsection{Proof of Proposition~\ref{bis}}
\label{charmap}

First, we use the argument of~\cite[Lemma~2.2]{PSadd} to prove Proposition~\ref{bis}\,b{\rm)}, i.e., to show that 
the natural pullback maps define isomorphisms $$A^*(P)\cong A^*(L)\cong A^*(C).$$
\begin{proof}[Proof of Proposition~\ref{bis}~b{\rm)}]
Since $P\rightarrow L$ is an affine fibration, and the inclusion $L\hookrightarrow P$ is its zero section, we have $A^*(P)\cong A^*(L)$. Next, for $T'=T\cap C$ and $B'=B\cap C$ consider the diagram
$$
\xymatrix{
A^*(\mathrm{B}T)\ar@{->>}[d]\ar[r]^{\!\!\mathfrak c}&\ar@{=}[d]\ar@{->>}[r]A^*(L/B)&A^*(L)\ar[d]\\
A^*\big(\mathrm{B}{T'}\big)\ar[r]^{\!\!\!\!\mathfrak c}&\ar@{->>}[r]A^*(C/B')&A^*(C),
}
$$
and use that the rows are exact.
\end{proof}

Next, we prove the part a) of Proposition~\ref{bis}, i.e., establish an isomorphism of rings
$$
A^*(P)\cong A^*(G)\otimes_{A^*(G/P)}A^*(\mathrm{pt}).
$$
The argument is essentially the same as in~\cite[Lemma~6.2]{PShopf}.

\begin{proof}[Proof of Proposition~\ref{bis}~a{\rm)}]
First, we consider the surjective map from Lemma~\ref{ck}\,a{\rm)} followed by the pullback map: 
$$
R^A\otimes_{A^*(\mathrm{pt})}V^A\twoheadrightarrow A^*(G/B)\rightarrow A^*(P/B),
$$
and tensor this sequence with $A^*(\mathrm{pt})$ over $A^*(G/P)$. We obviously have 
$$
A^*(\mathrm{pt})\otimes_{A^*(G/P)}R^A\otimes_{A^*(\mathrm{pt})}V^A\cong V^A,
$$
and since $P/B\rightarrow G/B\rightarrow G/P$ factors through a point, 
$$
A^*(\mathrm{pt})\otimes_{A^*(G/P)}A^*(P/B)\cong A^*(P/B).
$$ 
Therefore, we obtain the sequence
$$
V^A\twoheadrightarrow A^*(\mathrm{pt})\otimes_{A^*(G/P)}A^*(G/B)\rightarrow A^*(P/B)
$$
and the composite map is an isomorphism by Lemma~\ref{ck}\,b{\rm)}.

Now, we use $A^*(P/B)\cong A^*(L/B)$ and tensor the above sequence (of isomorphisms) with $A^*(\mathrm{pt})$ over $A^*(\mathrm BT)$. We have
$$
A^*(G/B)\otimes_{A^*(\mathrm BT)}A^*(\mathrm{pt})=A^*(G),
$$ 
and similarly $A^*(L/B)\otimes_{A^*(\mathrm BT)}A^*(\mathrm{pt})=A^*(L)$. Since the characteristic map commutes with pullbacks, we can conclude that the isomorphism 
$
{A^*(L)\cong A^*(G)\otimes_{A^*(G/P)}A^*(\mathrm{pt})}
$ 
is actually induced by the pullback along the inclusion $L\rightarrow G$ and, therefore, preserves the ring structure. We finish the proof with the use of Theorem~\ref{bis}~b{\rm)}.
\end{proof}

\subsection{BGG-Demazure operators and topological filtration}

Although the action of the BGG--Demazure operators $\Delta^\Omega_i$ on $\Omega^*(G/B)$ is quite complicated (see Subsection~\ref{BGGD}), they induce well-defined operators on the associated graded ring $\mathrm{gr}_\tau^*\Omega(G/B)^*$ of $\Omega^*(G/B)$ with respect to the topological filtration (see Subsection~\ref{filtration}). Moreover, the action of these induced operators can be described in terms of the action of the usual operators $\Delta^{\mathrm{CH}}_i$ on $\mathrm{CH}(G/B)$ (see Lemma~\ref{tech} below). This observation will be very helpful in the next subsection.

Recall that $\tau^1\Omega^*(G/B)=\mathrm{Ker}(\mathrm{deg}_\Omega)$ is an ideal in $\Omega^*(G/B)$ generated by the homogeneous elements of positive codimension, cf.~\cite[Theorem~1.2.14]{LM}, and $\tau^{s\,}\Omega^*(G/B)$ is an ideal generated by the homogeneous elements of codimension at least $s$, see Subsection~\ref{filtration}. Since the topological filtration respects the multiplication, clearly, $\tau^1\Omega^*(G/B)$ is nilpotent.

Below for a sequence of indices $I=(i_1,\ldots,i_k)$ we denote by $|I|=k$ its length, and for a homogeneous element $x\in\Omega^p(G/B)$ we denote by $|x|=p$ its codimension.
\begin{lm}
\label{tech}
The BGG--Demazure operators $\Delta^\Omega_I$ induce operators 
$$
\Delta^{\mathrm{gr}}_I\colon\mathrm{gr}_\tau^*\Omega^*(G/B)\rightarrow\mathrm{gr}_\tau^{*-|I|}\Omega^*(G/B), 
$$
and, moreover, the canonical map
$$
\rho\colon\mathrm{CH}^*(G/B)\otimes\mathbb L^*\rightarrow\mathrm{gr}_\tau^*\Omega^*(G/B)
$$
preserves the action of divided difference operators.
\end{lm}
\begin{proof}
We can assume that $I=\{i\}$. Then for a homogeneous $x\in\Omega^*(G/B)$ we have
$$
\Delta^\Omega_i(x)=-\pi_i^\Omega\circ(\pi_i)_{\Omega\,}(x)+\kappa_{i\,}x\in\tau^{|x|-1}\Omega^*(G/B),
$$
therefore, by the Leibniz rule~(\ref{leibniz}),  $\Delta^\Omega_i(ax)\in\tau^{|x|-1}\Omega^*(G/B)$ for every $a\in\Omega^*(G/B)$, i.e., 
$$
\Delta^\Omega_i\big(\tau^{p}\Omega^*(G/B)\big)\subseteq\tau^{p-1}\Omega^*(G/B).
$$
In particular, $\Delta^{\mathrm{gr}}_i$ is well-defined.

To obtain the second statement, take $x\in\mathrm{CH}^*(G/B)$ and let $\widetilde x$ be its lift in $\Omega^*(G/B)$. Recall that $\rho$ is $\mathbb L$-linear and commutes with pullbacks and pushforwards. Then $-\pi_i^\Omega\circ(\pi_i)_\Omega(\widetilde x)$ is a pre-image of $\Delta_i^{\mathrm{CH}}(x)$ in $\Omega^*(G/B)$, therefore $\rho$ sends $\Delta_i^{\mathrm{CH}}(x)$ to the class of this pre-image in $\mathrm{gr}_\tau^*\Omega^*(G/B)$. However, since $\kappa_i$ from~(\ref{sameBGGD}) in fact lies in $\mathbb L^{<0\,}\Omega(G/B)$, the classes of $-\pi_i^\Omega\circ(\pi_i)_\Omega(\widetilde x)$ and $\Delta_i^{\Omega}(\widetilde x)$ in $\mathrm{gr}_\tau^*\Omega^*(G/B)$ coincide.
\end{proof}

\subsection{Proof of Proposition~\ref{bisc}}
\label{bisc-section}

In the present subsection $n$ denotes a fixed natural number, $G=\mathrm{SO}_m$ with $m\geqslant 2^{n+1}+1$, and $P=P_1$ (see Subsection~\ref{schcal}). Then 
$$
Q=G/P
$$ 
is a smooth projective split quadric of dimension $m-2=2d$ or $m-2=2d+1$.

Let us briefly describe the plan of the proof realized below. We want to show that the natural map
$$
\mathrm K(n)^*(G/P)\rightarrow\mathrm K(n)^*(G)
$$
factors through $\mathrm K(n)^*(\mathrm{pt})$.

Let us denote $N=2^n$ for $m$ even and $N=2^n-1$ for $m$ odd. In our assumption $m\geqslant 2^{n+1}+1$, we have 
$N\leq d$. 
Recall from Proposition~\ref{mod2} that the ring $\mathrm K(n)^*(Q)$ is gene\-rated as an $\mathrm K(n)^*(\mathrm{pt})$-algebra by two elements: $h$ and $l=l_d$ connected by the following equation:
\begin{align}
\label{moravaformula}
h^{d+1}=v_n\,h^{N}l.
\end{align}
To obtain the claim it is sufficient to show that $h$ and $l$ are mapped to zero in $\mathrm K(n)^*(G)$.

Consider the natural map
$$
\pi\colon G/B\rightarrow G/P\cong Q,
$$
and let $\mathcal I=\mathcal I_{\mathrm K(n)}$ denote the kernel of the map
$$
\mathrm K(n)^*(G/B)\rightarrow\mathrm K(n)^*(G).
$$
Then the claim is equivalent to the fact that $H:=\pi^{\mathrm K(n)}(h)$ and $L:=\pi^{\mathrm K(n)}(l)$ lie in $\mathcal I$. However, $\mathcal I$ coincides with the ideal generated by (non-trivial) ``generically rational'' elements, i.e., the elements that lie in the image of the natural map
$$
\tau^1\mathrm K(n)^*(E/B)\rightarrow\mathrm K(n)^*(G/B),
$$
where $E\in\mathrm H^1(k,\,G)$ denotes the generic torsor (see Lemma~\ref{bt=egen} below for the precise statement). In particular, $H$ lies in $\mathcal I$, and it remains to show that $L$ also does.

With this end we apply $\pi$ to~(\ref{moravaformula}) and we claim that there exists a BGG-Demazure operator $\Delta_{I_0}^{\mathrm K(n)}$ such that $\Delta_{I_0}^{\mathrm K(n)}(H^{d+1})\in\mathcal I$ and $$\Delta_{I_0}^{\mathrm K(n)}(H^NL)\in a\cdot L+\mathcal I$$ for some invertible $a$.

More precisely, using the Leibniz rule~(\ref{leibniz}) we will show that
$$
\Delta_{I_0}^{\mathrm K(n)}(H^NL)\in \Delta_{I_0}^{\mathrm K(n)}(H^N)\cdot L+\mathcal I,
$$
and we will choose $I_0$ in such a way that $a=\Delta_{I_0}^{\mathrm K(n)}(H^N)$ is invertible.  This clearly implies that $L\in\mathcal I$.

Now we proceed with the proof of Proposition~\ref{bisc}.

For $A^*=\mathrm{CH}^*$ the pullback map along $\pi\colon G/B\rightarrow G/P$ can be described explicitly, 
in particular, 
$$
\pi^{\mathrm{CH}}(h^d)=
\begin{cases}
Z_{s_{d}\ldots s_1},&m=2d+3,\\
Z_{s_{d}\ldots s_1}+Z_{s_{d+1}s_{d-1}\ldots s_1},&m=2d+2;
\end{cases}
$$
and $\pi^{\mathrm{CH}}(h^k)=Z_{s_k\ldots s_1}$ for $0\leq k\leq d-1$. Indeed, in the notations of Section~\ref{schcal} the generator $h$ equals $Z_{s_1}$ and the pullback $\pi^{\CH}$ always maps $Z_w\in\CH^*(G/P)$ to $Z_w\in\CH^*(G/B)$ for every $w\in W^P$. Moreover, the Chow ring structure $\CH^*(G/P)$ is described in \cite[Propositions~68.1 and 68.2]{EKM} and the combinatorial structure of $W^P$ immediately follows from \cite[Example~9.2]{CGM}.

Recall that for $I=(i_1,\ldots,i_k)$ we denote by $\Delta^A_I=\Delta^A_{i_1}\circ\ldots\circ\Delta^A_{i_k}$ the composition of divided difference operators, and $\Delta^{\mathrm{CH}}_I=\Delta^{\mathrm{CH}}_w$ if $w=s_{i_1}\ldots s_{i_k}$ is a reduced decomposition (see Subsection~\ref{BGGD}). Using the duality of the Demazure operators and the Schubert basis~(\ref{bgg-schubert-duality}) we conclude that 
$$
\Delta^{\mathrm{CH}}_{I_0}\big(\pi^{\mathrm{CH}}(h^N)\big)=(-1)^NX_{w_0}=(-1)^N\in\mathrm{CH}^*(G/B)
$$
for $I_0=(1,2,\ldots,N)$.

\begin{rk}
In fact, we do not use explicit formulae for $\pi^{\mathrm{CH}}(h^N)$ and for $I_0$ in the subsequent arguments. We use only the existence of $I_0$ such that $\Delta^{\mathrm{CH}}_{I_0}\big(\pi^{\mathrm{CH}}(h^N)\big)$ is invertible.
\end{rk}

\begin{lm}
\label{invv}
In the above notation, for a free theory $A^*$ let us denote ${H_A=\pi^{A}(h)}$. Then the element $\Delta_{I_0}^A(H_A^{N})$ is invertible in $A^*(G/B)$.
\end{lm}
\begin{proof}
We can assume that $A^*=\Omega^*$. Since $\Delta_{I_0}^{\mathrm{CH}}(H_{\mathrm{CH}}^{N})=\pm1$, and $H_{\Omega}$ is clearly a preimage of $H_{\mathrm{CH}}$ in $\Omega^*(G/B)$, we conclude that $\Delta_{I_0}^\Omega(H_{\Omega}^{N})\in\pm1+\tau^1\Omega^*(G/B)$ by Lemma~\ref{tech}. Since $\tau^1\Omega^*(G/B)$ is nilpotent, it follows that $\Delta_I^\Omega(H_{\Omega}^{N})$ is invertible. 
\end{proof}

\begin{nt}[Ideal $\mathcal I_A$]
\label{ia}
For any free theory $A^*$ consider the characteristic map
$$
\mathfrak c_A\colon A^*(\mathrm BT)\rightarrow A^*(G/B),
$$
and let $\mathcal I_A=\mathcal I_A(G/B)$ denote the ideal of $A^*(G/B)$ generated by the set $\mathrm{Im}(\mathfrak c_A)\cap\mathrm{Ker}(\mathrm{deg}_A)$. 
\end{nt}

Let $E$ be a generic $G$-torsor, see Section~\ref{torsor}, and let us denote by $\overline{A^*}(E/B)$ the image of the restriction map 
$$
A^*(E/B)\rightarrow  A^*(G/B).
$$ 
We call elements of this set {\it rational}. Since $\Delta_i$ can be defined on $A^*(E/B)$ by the same formula~(\ref{sameBGGD}) as for $A^*(G/B)$, we conclude that $\Delta_i$ applied to a rational element is again rational.

The following result is obtained in~\cite[Lemma~5.3]{PShopf} (cf. also~\cite[Example~4.7]{PShopf}).

\begin{lm}
\label{bt=egen}
In the notation above, the ideal $\mathcal I_A$  coincides with the ideal generated by the set
$
{\,\overline{A^*}(E/B)\cap\mathrm{Ker}(\mathrm{deg}_A)}
$.
\end{lm}

Now we can prove the following
\begin{lm}
\label{non-danger}
In the above notation, for any free theory $A^*$, any set of indices $I$, and any $s>|I|$ the element $\Delta^A_I\big(\pi^A(h^s)\big)$ lies in $\mathcal I_A$. 
\end{lm}
\begin{proof}
We can assume that $A^*=\Omega^*$. Since $h=c^\Omega_1\big(\mathcal O_Q(1)\big)$ we conclude that $\pi^\Omega(h^s)$ is rational, and, therefore, $\Delta^\Omega_I\big(\pi^\Omega(h^s)\big)$ is rational as well. On the other hand, since $s>|I|$, we see that $\Delta^\Omega_I\big(\pi^\Omega(h^s)\big)$ lies in $\mathrm{Ker}(\mathrm{deg}_\Omega)=\tau^1\Omega^*(G/B)$ by Lemma~\ref{tech}.
\end{proof}

Finally, we obtain the following
\begin{lm}
\label{last}
Let $A^*$ be a free theory, and let us denote $L=\pi^A(l)$ and $H=\pi^A(h)$. Then 
$$
\Delta^A_{I_0}(H^{N}L)\in\Delta^A_{I_0}(H^{N})L+\mathcal I_A.
$$

\end{lm}
\begin{proof}
In the proof we write $\Delta$ for $\Delta^A$.
We show by induction on $|I|\leq N$ that $\Delta_{I}(H^{N}L)-\Delta_I(H^{N})L$ is a sum of elements of the following types:
\begin{enumerate}
\item
an $A^*(G/B)$-multiple of $\mathfrak c_A(x_\lambda)\Delta_I(H^{N})$ for some $\lambda\in M$;
\item
an $A^*(G/B)$-multiple of $\Delta_{I'}(H^{N})$ for $|I'|<|I|$.
\end{enumerate}
Since $\mathfrak c_A(x_\lambda)\in\mathcal I_A$, with the use of Lemma~\ref{non-danger} we obtain the claim.

For the induction step we apply $\Delta_i$ to each summand. First, for $a\in A^*(G/B)$ we have
$$
\Delta_i\big(a\,\mathfrak c_A(x_\lambda)\Delta_I(H^N)\big)=\Delta_i\big(a\,\mathfrak c_A(x_\lambda)\big)\Delta_I(H^N)+s_i\big(a\,\mathfrak c_A(x_\lambda)\big)(\Delta_i\circ\Delta_I)(H^N),
$$
where $s_i\big(a\,\mathfrak c_A(x_\lambda)\big)=s_i(a)\,\mathfrak c_A(x_{s_i(\lambda)})$, and we obtain summands of two types. 

Next,
$$
\Delta_i\big(a\,\Delta_{I'}(H^N)\big)=\Delta_i(a)\Delta_{I'}(H^N)+s_i(a)\,(\Delta_i\circ\Delta_{I'})(H^N),
$$
i.e., we obtain two summands of the second type.

Finally, we have to apply $\Delta_i$ to $\Delta_{I}(H^{N}L)-\Delta_I(H^{N})L$. Clearly, $$\Delta_i\big(\Delta_I(H^{N})L\big)=(\Delta_i\circ\Delta_I)(H^{N})\,L+s_i\big(\Delta_I(H^{N})\big)\Delta_i(L),$$ where
$$
s_i\big(\Delta_I(H^{N})\big)=\Delta_I(H^{N})-\mathfrak c_A(x_{\alpha_i})\big((\Delta_i\circ\Delta_I)(H^{N})\big),
$$
and the proof is finished.
\end{proof}

Now we can finish the proof of Proposition~\ref{bisc}, i.e., we can show that the map
$$
\mathrm K(n)^*(Q)\rightarrow\mathrm K(n)^*(\mathrm{SO}_m)
$$
factors through $\mathrm K(n)^*(\mathrm{pt})$.

\begin{proof}[Proof of Proposition~\ref{bisc}]
Let $h$, $l=l_d\in\mathrm K(n)^*(Q)$ be as in Proposition~\ref{mod2}, and $\pi\colon G/B\rightarrow Q$ denote the natural projection. In the notation above, we will prove that $H=\pi^{\mathrm K(n)}(h)$ and $L=\pi^{\mathrm K(n)}(l)$ belong to $\mathcal I_{\mathrm K(n)}$. Indeed, $H\in\mathcal I_{\mathrm K(n)}$ by Lemma~\ref{non-danger}. Next, applying $\pi^{\mathrm K(n)}$ to~(\ref{moravaformula}) we obtain the equation
$$
H^{d+1}=v_nH^NL.
$$
On the one hand, $\Delta_{I_0}^{\mathrm K(n)}(H^{d+1})$ lies in $\mathcal I_{\mathrm K(n)}$ by Lemma~\ref{non-danger}. On the other hand, 
$$
\Delta_{I_0}^{\mathrm K(n)}(H^{N}L)\in a\cdot L+\mathcal I_{\mathrm K(n)}
$$ 
for some invertible $a\in\mathrm K(n)^*(G/B)$ by Lemmata~\ref{last} and~\ref{invv}. This implies that $L$ lies in $\mathcal I_{\mathrm K(n)}$. 

Recall that $\mathcal I_{\mathrm K(n)}$ coincides with the kernel of the map
$$
\mathrm K(n)^*(G/B)\rightarrow\mathrm K(n)^*(G)
$$
by~(\ref{Kr12}). Then $h$, $l\in\mathrm K(n)^*(Q)$ go to $0$ in $\mathrm K(n)^*(G)$. However, $\mathrm K(n)^*(Q)$ is generated by $h$ and $l$ as algebra (see Proposition~\ref{mod2}), and, therefore, the map
$$
\mathrm K(n)^*(Q)\rightarrow\mathrm K(n)^*(\mathrm{SO}_m)
$$
factors through $\mathrm K(n)^*(\mathrm{pt})$ as claimed.
\end{proof}

\begin{proof}[Proof of Theorem~\ref{stab}]
Let $n$ be a natural number, $G=\mathrm{SO}_m$ for $m\geq2^{n+1}+1$, $P=P_1$ the maximal parabolic subgroup corresponding to the first root, $L$ the corresponding Levi subgroup (see Section~\ref{schcal}). Then $Q=G/P$ is a split $m-2$ dimensional quadric, and $C=[L,\,L]$ is isomorphic to $\mathrm{SO}_{m-2}$.

Using Proposition~\ref{bis} we conclude that
$$
\mathrm K(n)^*(\mathrm{SO}_{m-2})\cong\mathrm K(n)^*(P)\cong\mathrm K(n)^*(\mathrm{SO}_m)\otimes_{\mathrm K(n)^*(G/P)}\mathrm K(n)^*(\mathrm{pt})
$$
where the isomorphism is induced by the pullback along the natural inclusion.

On the other hand, by Proposition~\ref{bisc} we conclude that the natural map $\mathrm K(n)^*(Q)\rightarrow\mathrm K(n)^*(\mathrm{SO}_m)$ factors through $\mathrm K(n)^*(\mathrm{pt})$, i.e., $$
\mathrm K(n)^*(\mathrm{SO}_m)\otimes_{\mathrm K(n)^*(Q)}\mathrm K(n)^*(\mathrm{pt})\cong\mathrm K(n)^*(\mathrm{SO}_m)\otimes_{\mathrm K(n)^*(\mathrm{pt})}\mathrm K(n)^*(\mathrm{pt}).
$$
\end{proof}

\section{Computation of the Morava K-theory ring for special orthogonal groups}
\label{smallso}

In this section as previously we denote by $\kn -$ the Morava K-theory modulo~$2$, in particular, $\kn\pt=\mathbb F_2[v_n^{\pm1}]$. We will compute the ring $\kn\so$.

\subsection{Chow ring of the special orthogonal group}
\label{chowso}

The structure of the Chow ring of the group variety $\mathrm{SO}_m$ modulo $2$ is well-known
 (see~\cite[Table~II]{Kac}):
$$
\mathrm{CH}^*(\mathrm{SO}_m;\,\mathbb F_2)=\mathbb F_2[e_1,e_2,\ldots,e_s]/(e_i^2-e_{2i}),
$$
where $m=2s+1$ or $m=2s+2$, $e_i$ has degree $i$, and $e_{2i}$ stands for $0$  if $2i>s$.
The above identity can be written in a slightly different form: 
$$
\mathrm{CH}^*(\mathrm{SO}_m;\,\mathbb F_2)=
\mathbb F_2[e_1,e_3,\ldots,e_{2r-1}]/(e_{2i-1}^{2^{k_i}}),
$$
where $r=\lfloor\frac{m+1}{4}\rfloor$, $k_i=\lfloor\mathrm{log}_2(\frac{m-1}{2i-1})\rfloor$.  
Moreover, the multiplication of $\mathrm{SO}_m$
induces a Hopf algebra structure on $\mathrm{CH}^*(\mathrm{SO}_m;\,\mathbb F_2)$
and the elements $e_{2i-1}$ are primitive (or Lie-like) for the comultiplication $\Delta$,
 i.e., $\Delta(e_{2i-1})=1\otimes e_{2i-1}+e_{2i-1}\otimes1$.

These notations do not contradict each other, and it is sometimes convenient to write $e_{2^k(2i-1)}$ instead of $e_{2i-1}^{2^k}$.

\subsection{Hopf theoretic lemmata}
\label{hopftheoreticlemmata}

The results of this section are based on the general statements about Hopf algebras.

Consider a graded Hopf algebra over $\F_{p}$:
$$
H= \F_{p}[x_1,\ldots, x_r]/(x_1^{p^{k_1}}, \ldots , x_r^{p^{k_r}})
$$
where $x_i$ are primitive of positive degree, $k_i>0$.
By an {\it index set} in this section we mean an $r$-tuple $I=(d_1,\ldots,d_r)$ 
with $0\leq d_i<p^{k_i}$. For an index set $I$ let $x_I$ denote $\prod_{i=1}^rx_{i}^{d_i}$,
thus, $\{x_I\mid I\text{ an index set\,}\}$ is an $R$-basis of $H\otimes_{\mathbb F_p}R$
 over any $\mathbb F_p$-algebra $R$. 
 We call it {\it the standard basis}. 
 
The following technical lemma is used in the proof of Proposition~\ref{abstracthopf}.

\begin{Lm}
\label{technicalhopf}
In the above notation 

\noindent
a{\rm)} All primitive elements of $H$ are equal to $\sum\limits_hx_{j_h}^{p^{q_h}}$ for some $j_h$ and $q_h<k_{j_h}$.\\
b{\rm)} Assume that for $P\in H$, $Q\in H\otimes_{\mathbb F_p} H$ and $k\in\mathbb N$ the equality $$\Delta(P)=P \otimes 1 + 1\otimes P + Q^{\,p^{k}}$$ holds. Then
$P-S^{\,p^{k}}$ is primitive for some $S\in H$.
\end{Lm}
\begin{proof}
Claim ``a)'' follows from the description of primitive elements in the tensor product of Hopf algebras and the easy case $r=1$. In fact, it also follows from the argument below.

To obtain ``b)'' decompose $P$ in the standard basis $x_I$, and let $f$ be the minimal possible non-negative integer such that for some tuple $I=(d_1,\ldots,d_r)$ with $d_i=p^{f}q$, $p\nmid q$, the element $x_I$ 
 appears in the decomposition of $P$. We can assume that $f<k$, since otherwise $P=S^{p^k}$ for some $S$ and the claim follows.
  
 The primitivity of $x_i$ implies that $\Delta(x_I)$ has the monomial $x_{i}^{p^{f}}\otimes x_{J}$ with a non-zero coefficient in its decomposition for some $J$. Moreover, this monomial does not appear in the decomposition of $\Delta(x_{I'})$ for $I\neq I'$.
 
 However, $P \otimes 1 + 1\otimes P + Q^{\,p^{k}}$ cannot contain an element $x_{i}^{p^{f}}\otimes x_{J}$ in its decomposition for $x_J\neq1$, since $f<k$. This allows us to conclude that $x_I=x_{i}^{p^{f}}$. Then we can change $P$ by $P-x_{i}^{p^{f}}$ and continue by induction.
\end{proof}

The following purely Hopf theoretic proposition is a key ingredient in our argument.

\begin{Prop}
\label{abstracthopf}

Let $H$ be a $\ZZ$-graded Hopf algebra over $\F_{p}[v_n]$ which is free and finitely generated as an $\F_p[v_n]$-module {\rm(}as usual, $v_n$ has degree $1-p^n${\rm)}.
Assume that $H/v_n$ has 
homogeneous generators $\overline x_1, \ldots, \overline x_r$ of positive degrees such that they all are primitive for the Hopf algebra structure and, moreover, there exists an isomorphism
$$
H/v_n\cong \F_{p}[\overline x_1,\ldots, \overline x_r]/(\overline x_1^{\,p^{k_1}}, \ldots , \overline x_r^{\,p^{k_r}})
$$
for some positive numbers $k_i$. Assume furthermore that for every pair of numbers $i$, $j$ from $1$ to $r$ {\rm(}in particular, these may be equal{\rm)} the equation
\begin{align}
\label{nosolutions}
(1-p^n)x+\deg \overline x_j\cdot p^y = \deg \overline x_i\cdot p^{k_i} 
\end{align}
has no solutions for $x>0$, $0\le y<k_j$, $x$, $y\in\mathbb Z$. Then there exist homogeneous lifts $x_i$ of the variables 
$\overline x_i$ to $H$ such that  $$H\cong \F_{p}[v_n][x_1,\ldots, x_r]/(x_1^{p^{k_1}}, \ldots, x_r^{p^{k_r}})$$ as an algebra.

Moreover, assume that for some $i$ there exists $\widetilde x_i\in H$ 
which lifts $\overline x_i$ and such that $\Delta(\widetilde x_i)$ can be written as
an $\mathbb{F}_p[v_n]$-linear combination of
$\widetilde x_i^{\,b}\otimes \widetilde x_i^{\,c}$, $b$, $c\ge 0$. Then we can choose $x_i=\widetilde x_i$.
\end{Prop}

\begin{proof}
Let $\overline H$ denote $H/v_n$, and if $h$ is an element of $H$ or of $H\otimes_{\mathbb F_p[v_n]}H$,
then let $\overline h$ denote the image of $h$ modulo $v_n$ in $\overline H$ or, respectively, in $\overline  H\otimes_{\mathbb F_p}\overline  H$.

First, let us choose arbitrary lifts of $\overline x_i$ to $H$ 
and denote them $x_i$. We will modify these generators of $H$
 to obtain relations of the correct form.

By the assumptions 
for every $i$ and for some $a_i>0$ and $Q_i\in H$ a relation of the following form holds in $H$:
\begin{align}
\label{list}
x_i^{p^{k_i}} = (v_n)^{a_i}\,Q_i.
\end{align}

We will describe a procedure that at each step chooses a new lift of $\overline x_i$ so that the value $a_i$ in the relation~(\ref{list}) increases. Since the values $a_i$ are bounded above by the fact that $H^k=0$ for $k>\!>0$, the procedure will stop after a finite number of steps, i.e., $x_i^{p^{k_i}}$ for every $i$ will become equal to zero.

When this is achieved, we claim that $H\cong \F_{p}[v_n][x_1,\ldots, x_r]/(x_1^{p^{k_1}}, \ldots x_r^{p^{k_r}})$. 
Let $R\subset \mathbb F_p[v_n][x_1,\ldots, x_r]$ be the ideal of the relations of $H$.
Since $H$ is a free $\mathbb F_p[v_n]$-module, it follows that
 $R/v_n=(\bar{x}_1^{p^{k_1}}, \ldots \bar{x}_r^{p^{k_r}})$ 
 and we have shown that $(x_1^{p^{k_1}}, \ldots x_r^{p^{k_r}})\subset R$.
Therefore, $\left(R/(x_1^{p^{k_1}}, \ldots x_r^{p^{k_r}})\right)/v_n$ is zero, which implies
$R=(x_1^{p^{k_1}}, \ldots x_r^{p^{k_r}})$.

To choose a different lift of $\overline x_i$ 
we make a ``change of variables''
\begin{align}
\label{change}
x_i^{\mathrm{new}} = x_i + v_n S,
\end{align}
for some $S\in H$, so that $x_i^{\mathrm{new}}$ satisfies the relation $(x_i^{\mathrm{new}})^{p^{k_i}} = (v_n)^{a_i^{\mathrm{new}}}\,Q_i^{\mathrm{new}}$. Our goal is to find $S$ such that $a_i^{\mathrm{new}}>a_i$.
After we have performed this change, we will denote $x_i^{\mathrm{new}}$ just as $x_i$ 
by abuse of notation.

Let us choose a list of relations among~(\ref{list}), one for every $x_j$. Furthermore,
we can assume that for every $j$ element $\overline Q_j$ is non-zero, i.e., $Q_j$ is not divisible by $v_n$.
Let $i$ be the index of a variable $x_i$ with minimal $a_i=:a$.
The relation in our list for this $i$ can be written as follows: 
\begin{align}
\label{starting}
x_i^{p^{k_i}} = v_n^a P + v_n^{a+1}R
\end{align}
where $a\ge 1$, $R\in H$ and $P$ is a linear combination of $x_I$ with $\F_{p}$-coefficients

Let $Q_t$ be linear combinations of $x_I\otimes x_J$ with $\F_{p}$-coefficients
such that $$\Delta(x_i) = x_i \otimes 1 + 1\otimes x_i + \sum_{t\ge t_0} v_n^t Q_t,$$ $t_0\ge 1$.
Then we have 
\begin{equation}\label{eq:power-Delta}
\Delta(x_i)^{p^{k_i}} = (x_i)^{p^{k_i}} \otimes 1 + 1\otimes (x_i)^{p^{k_i}}
 + \sum_{t\ge t_0} v_n^{t\cdot p^{k_i}} (Q_t)^{p^{k_i}}.
\end{equation}
We now rewrite the sum containing $(Q_t)^{p^{k_i}}$ using relations between the variables from the chosen list of relations. 
If for some $t\geq t_0$ we have that $(\overline{Q_{t}})^{p^{k_i}}=0$,
then we claim that $(Q_{t})^{p^{k_i}}=v_n^a\,Q'$  for $Q'\in H\otimes_{\mathbb F_p[v_n]} H$. 
Indeed,
if we express $Q_{t}$ as a linear combination of $x_I\otimes x_J$ with $\F_{p}$-coefficients,
then $(Q_{t})^{p^{k_j}}$ goes to zero in $\overline H\otimes_{\mathbb F_p}\overline H$ 
only if for all non-trivial summands $x_I\otimes x_J$ 
of $Q_t$ either $\overline x_I^{p^{k_j}}$ or $\overline x_J^{p^{k_j}}$  is zero in $\overline H$.
However, this means precisely that $x_l^q$ with $q>p^{k_l}$ appears in either $x_I^{p^{k_j}}$ or $x_J^{p^{k_j}}$.
Therefore, we can apply one of the relations~(\ref{list}) from the chosen list
and using the fact that $a\le a_l$ by our choice, we can replace $(Q_{t})^{p^{k_j}}$
 by $v_n^a Q'$ for some $Q'$.

Using this and plugging in relation~(\ref{starting}) we can rewrite~(\ref{eq:power-Delta})
for some $t_1\geq t_0$ and some $R'$ as follows:

$$
\Delta\big(x_i^{p^{k_i}}\big) = v_n^a\left(P \otimes 1 + 1\otimes P\right)
 + \sum_{t\ge t_1} v_n^{t\cdot p^{k_i}} (Q_t)^{p^{k_i}}
 + v_n^{a+1}R'
$$ 
where $(\overline{Q_{t_1}})^{p^{k_i}}$ is not zero in $\overline H\otimes_{\mathbb F_p}\overline  H$. 

On the other hand, from~(\ref{starting}) we have 
$$
\Delta\big(x_i^{p^{k_i}}\big) =\Delta (v_n^a P + v_n^{a+1} R) = v_n^a\, \Delta(P) + v_n^{a+1} \Delta(R).
$$ 
Comparing these two expressions of $\Delta\big(x_i^{p^{k_i}}\big)$ we see that $a\leq t_1 p^{k_i}$,
as otherwise $Q_{t_1}^{\,p^{k_i}}$ needs to be divisible by $v_n$ contradicting our assumptions.

We show now 
that $a<t_1 p^{k_i}$ cannot happen. 
Indeed, in this case we would obtain that $\overline P$ is primitive (and non-zero by the assumption).
Moreover, relation~(\ref{starting}) gives us that $p^{k_i} \deg (x_i) = \deg P -(1-p^n)a$.
However, all primitive elements in $\overline H$ are of the form $\sum_h\overline x_{j_h}^{p^{q_h}}$ for some $j_h$ and $q_h<k_{j_h}$ by Lemma~\ref{technicalhopf}\,a), and, therefore, the equation on the degrees can be rewritten
as $p^{k_i} \deg (x_i) = (p^n-1)a+p^{s_h}(\deg x_{j_h})$ for each $h$. This equation has no solutions by our assumption~(\ref{nosolutions}).

Let us now consider the case $a=t_1 p^{k_i}$. Again comparing two expressions of $\Delta\big(x_i^{p^{k_i}}\big)$
we see that $\Delta(\overline P)$ in the quotient Hopf algebra 
has the form ${\overline P \otimes 1 + {1\otimes\overline P} {+ \,\overline {Q_{t_1}^{\,p^{k_i}}} }}$.
By Lemma~\ref{technicalhopf}\,b) it follows that $\overline P- \overline S^{\,p^{k_i}}$ is primitive for some $\overline S\in\overline H$. However, assumption~(\ref{nosolutions}) again implies that such a primitive element equals zero, i.e., $\overline P= \overline S^{\,p^{k_i}}$. 
Now take any lift $S\in H$ of $\overline S$ and define  $x_i^{\mathrm{new}} = x_i -v_n^{\,t_1} S$. 
Then relation~(\ref{starting}) can be rewritten as
$$ (x_i^{\mathrm{new}})^{p^{k_i}} = v_n^a (P-S^{p^{k_i}})+v_n^{a+1}R
=v_n^{a+1}  {Q_i}^{\mathrm{new}}$$
for some ${Q_i}^{\mathrm{new}}$, so that we have performed the required modification of the variables. 

To finish the proof of the proposition we need to explain that if we choose 
$x_i$ such that $\Delta(x_i)$ is an $\mathbb{F}_p[v_n]$-linear combination of $x_i^b\otimes x_i^c$, $b$, $c\ge 0$,
then the procedure described above does not modify $x_i$.
Indeed, in this case every $Q_t$ in~\eqref{eq:power-Delta} is divisible by $x_i\otimes x_i$ 
(i.e., in every monomial $x_i^b\otimes x_i^c$ we have $b$, $c\ge 1$), 
and, therefore, rewriting~\eqref{eq:power-Delta} as described above we will get 
$t_1p^{k_i}\geq2a>a$. In particular, 
the case $a=t_1p^{k_i}$ never happens, 
and the variable $x_i$ is not modified by our procedure.
\end{proof}

\begin{Rk}
The following example shows that there exists a graded Hopf algebra $H$ that is 
 free and of finite rank as an $\F_{2}[v_n]$-module for $v_n$ of negative degree and 
 does not satisfy~(\ref{nosolutions}) and the claim of Proposition~\ref{abstracthopf}.  Let $n=1$, $p=2$, 
 and consider the following Hopf algebra:

$$H:= \F_2[v_1][x_1,x_3]/(x_1^2-v_1 x_3, x_3^4),$$
where $\deg v_1 = -1$, $\deg x_i = i$
and the comultiplication is given by the following formulae
$$ \Delta (x_1) = x_1 \otimes 1 + 1\otimes x_1 + v_1 x_1 \otimes x_1,\quad
\Delta (x_3) = x_3 \otimes 1 + 1\otimes x_3 + v_1^3 x_3 \otimes x_3.$$
It is easy to check that $H/v_1$ and $H/(v_1-1)$ are not isomorphic as algebras.
In other words, condition~(\ref{nosolutions}) in Proposition~\ref{abstracthopf} cannot be omitted.
\end{Rk}

\subsection{Associated graded ring of the connective Morava K-theory}

In this section we consider the connective Morava K-theory modulo $p$ and denote it by $\ckn -$. In particular, $\ckn\pt=\mathbb F_p [v_n]$ and $\deg v_n=1-p^n$. We will work with the topological filtration $\tau^*$ on the connected Morava K-theory, see Subsection~\ref{filtration}. We remark that each associated graded piece of the topological filtration is itself graded, so that $\grckn -$ is a bi-graded theory, and after a change of grading it coincides with the graded ring associated with the filtration on $\ckn -$ by powers of $v_n$,
$$
\grckn -=\bigoplus_{i\geq0}\,\frac{v_n^i\,\ckn -}{v_n^{i+1}\,\ckn -},
$$
cf. the analogous statement on algebraic cobordism in~\cite[Theorem~4.5.7]{LM} and~\cite[Section~4]{Vcob}.

Recall that a torsion in a graded $\mathbb F_p[v_n]$-module $M=(M_i)_{i\in\mathbb Z}$ can only be a $v_n^\mathbb{Z}$-torsion.
This simple fact implies the following result.

\begin{Prop}
\label{onlytorsion}
Let $G$ be a split reductive group.\\ 
i{\rm)} The connective Morava K-theory of $G$ has the following {\rm(}non-canonical{\rm)} form as a graded $\mathbb F_p [v_n]$-module:
$$
\mathrm{CK}(n)^*(G)=\bigoplus_i\mathbb F_p[v_n]\cdot x_i\oplus\bigoplus_j\big(\mathbb F_p[v_n]/v_n^{d_j}\big)\cdot y_j,
$$
for some homogeneous $x_i$, $y_j\in\mathrm{CK}(n)^*(G)$, $d_j\geq0$ {\rm(}and the direct sums are finite{\rm)}. If $\overline x_i$, $\overline y_j$ denote the classes of $x_i$, $y_j$ in the associated graded ring $\mathrm{gr}_\tau^*\mathrm{CK}(n)^*(G)$, then
$$
\mathrm{gr}_\tau^*\mathrm{CK}(n)^*(G)=\bigoplus_i\mathbb F_p[v_n]\cdot \overline x_i\oplus\bigoplus_j\big(\mathbb F_p[v_n]/v_n^{d_j}\big)\cdot \overline y_j.
$$
ii{\rm)} We can identify $\mathrm{CH}^*(X;\,\mathbb F_p[v_n])$ with
$$
\bigoplus_i\mathbb F_p[v_n]\cdot x_i\oplus\bigoplus_j\mathbb F_p[v_n]\cdot y_j.
$$
Then the canonical map
\begin{align}
\label{chgr}
\rho\colon\mathrm{CH}^*(G;\,\mathbb F_p[v_n])\twoheadrightarrow\mathrm{gr}_\tau^*\mathrm{CK}(n)^*(G)
\end{align}
from~\cite[Corollary~4.5.8]{LM} 
sends $x_i$ to $\overline x_i$ and $y_j$ to $\overline y_j$.\\
iii{\rm)} The Hopf algebra structure on $\mathrm{CK}(n)^*(G)$ induces a Hopf algebra structure on $\mathrm{gr}_\tau^*\mathrm{CK}(n)^*(G)$, and the map~{\rm($\!\,^{\!}$\ref{chgr}\rm)} is a morphism of Hopf algebras.
\end{Prop}

\begin{proof}
 {\it i}{\rm)} Since $\mathrm{CK}(n)^*(G)$ is finitely generated over $\mathbb F_p[v_n]$, e.g., by~\eqref{Kr12}, the first claim of $i$) follows from the remark on torsion above. To get the second claim, recall that the topological filtration on $\mathrm{CK}(n)^*(G)$ coincides with the filtration by the powers of $v_n$.\\ 
{\it ii}{\rm)} Recall that $\mathrm{CH}^*(G;\,\mathbb F_p[v_n]) = \mathrm{CK}(n)^*(G)/v_n \otimes_{\mathbb F_p} \mathbb F_p[v_n]$, 
and the map $\rho$ is $\mathbb F_p[v_n]$-linear and sends an element of $\mathrm{CH}^*(G;\,\mathbb F_p)$ to the class of its lift in $\mathrm{CK}(n)^*(G)$.\\
{\it iii}{\rm)} The comultiplication on $\mathrm{gr}_\tau^*\mathrm{CK}(n)^*(G)$ induced from $\mathrm{CK}(n)^*(G)$ coincides with 
the comultiplication induced from $\mathrm{CH}^*(G;\,\mathbb F_p[v_n])$ via $\rho$. 
Since the latter defines a Hopf algebra structure, $iii$) follows.
\end{proof}

\subsection{Computation of the associated graded ring}
\label{ass-gr-so}

In the notation of the previous section, we assume additionally that $p=2$ and $G=\so$.

In this section we will compute $\grckn\so$ and show that $\ckn\so$ is a finitely generated free $\F_2[v_n]$-module for $m\le 2^{n+1}$. Luckily, this is precisely the range of $m$ when the $n$-th Morava K-theory of $\mathrm{SO}_m$ is not stabilised, see Theorem~\ref{stab}.

The key ingredient in the proof of this result is Vishik's theorem on the relations in algebraic cobordism~\cite[Theorem~4.3]{Vcob}, see Lemma~\ref{es} below. 

Recall that Proposition~\ref{bis} implies that
$$
\mathrm{CK}(n)^*(\mathrm{SO}_{m-2})\cong\mathrm{CK}(n)^*(\mathrm{SO}_{m})\otimes_{\mathrm{CK}(n)^*(Q_{m-2})}\mathbb F_2[v_n],
$$
for $Q_{m-2}$ a split $(m-2)$-dimensional quadric $\mathrm{SO}_{m}/P_1$, and the isomorphism is induced by the pullback along the natural inclusion of $\mathrm{SO}_{m-2}$ into $\mathrm{SO}_{m}$. Moreover, Lemma~\ref{bt=egen} implies that the image of $h\in\mathrm{CK}(n)^*(Q_{m-2})$ is trivial in $\mathrm{CK}(n)^*(\mathrm{SO}_{m})$, and, therefore, the kernel of the pullback map 
$$
\mathrm{CK}(n)^*(\mathrm{SO}_{m})\twoheadrightarrow\mathrm{CK}(n)^*(\mathrm{SO}_{m-2})
$$
is generated (as an ideal) by a single element $y\in\mathrm{CK}(n)^*(\mathrm{SO}_{m})$, more precisely, by the image of $l\in\mathrm{CK}(n)^*(Q_{m-2})$
where $l$ is the class of an $\lfloor(m-2)/2\rfloor$-dimensional linear subspace.
 Since $l^2$ is divisible by $h$, we conclude that $y^2=0$.

\begin{Lm}
\label{es}
In the above notation let $m\leq2^{n+1}$. Then the image $\overline y$ of $y$ in $\mathrm{gr}_\tau^*\mathrm{CK}(n)^*(\mathrm{SO}_{m})$ is not a $v_n^{\mathbb Z}$-torsion.
\end{Lm}
\begin{proof}
Denote $G=\mathrm{SO}_m$. It follows from~\cite[Theorem 4.3]{Vcob} that $\mathrm{CK}(n)^*(G)=\Omega^*(G)\otimes_{\mathbb{L}} \mathbb{F}_2[v_n]$ has generator $1$ in degree $0$ and other generators in positive degrees with relations in positive degrees {\rm(}as $\mathbb{F}_2[v_n]$-module{\rm)}. Moreover, the latter generators can be chosen to be the lifts of any homogeneous $\mathbb F_2$-basis of $\mathrm{CH}^*(G;\,\mathbb F_2)$, see~\cite[Proof of Theorem~4.3]{Vcob}.

In particular, an element $x\in\mathrm{CK}(n)^i(G)$, $i\le 2^n-1$, projecting to a non-trivial element in $\mathrm{CH}^i(G;\,\mathbb F_2)$, cannot be a $v_n^{\mathbb Z}$-torsion.

Next, recall that the topological filtration on $\mathrm{CK}(n)^*(G)$ coincides with the filtration by the powers of $v_n$. 
Thus, if $\overline{y}\in\mathrm{gr}^*_\tau\mathrm{CK}^*(n)$ is a $v_n^{r}$-torsion for some $r$, 
then $v_n^r y = v_n^{r+1}z$ for some $z\in\mathrm{CK}(n)^*(G)$.  
However, this implies that $v_n^r (y-v_nz) = 0$ contradicting the above observation (cf. Proposition~\ref{onlytorsion}\,{\it i}).

\end{proof}

\begin{Rk}
Alternatively, the claim that $\overline y$ is not a $v_n^{\mathbb Z}$-torsion
can be obtained using \cite[Proposition~6.2]{Se2}. 
Note that both \cite{Se2} and \cite{Vcob} are based on Vishik's classification of operations
in \cite{V12}, so that it is hardly a different argument at the end of the day. 
\end{Rk}

Now we are ready to compute $\grckn\so$ for $m\leq2^{n+1}$.

\begin{Prop}\label{prop:graded_isom}
For $m\le 2^{n+1}$ the canonical surjective morphism 
$$
\rho:\mathrm{CH}^*(\so;\,\mathbb F_2[v_n])\twoheadrightarrow\grckn\so
$$ 
is an isomorphism.
\end{Prop}

\begin{proof}

We prove the claim by induction on $m$. The base of induction $m=1$, $2$ is clear.

Consider the following commutative diagram with vertical pullback maps:

$$
\begin{tikzcd}
\mathrm{CH}^*(\mathrm{SO}_{m};\,\mathbb F_2[v_n])\ar[two heads]{r}{\rho_{m}}\ar[two heads]{d}&\grckn{\mathrm{SO}_{m}}\ar[two heads]{d}\\
\mathrm{CH}^*(\mathrm{SO}_{m-2};\,\mathbb F_2[v_n])\ar{r}{\rho_{m-2}}[swap]{\sim}&\grckn{\mathrm{SO}_{m-2}}
\end{tikzcd}
$$

The kernels of the vertical maps are generated by elements $e_s$ and $\overline y$, respectively,  
where ${s=\lfloor\frac{m-1}{2}\rfloor}$, see Subsection~\ref{chowso}.
We can assume that $e_s$ coincides with the image of $l\in\mathrm{CH}^*(Q_{m-2};\,\mathbb F_2[v_n])$ in $\mathrm{CH}^*(\mathrm{SO}_{m};\,\mathbb F_2[v_n])$ under the natural pullback map, where $Q_{m-2}=\mathrm{SO}_{m}/P_1$ is the projective quadric of dimension $m-2$ (cf.~\cite[Proof~of~Lemma~7.2]{PShopf}). 

From the commutativity of the above diagram one can see that the kernel of $\rho_{m}$ is contained in the ideal $(e_s)$. Then it is easy to check that $\rho_{m}$ is injective on primitive elements. 

Indeed, by Lemma~\ref{technicalhopf} and by the computation of Chow rings of $\grckn{\mathrm{SO}_{m}}$ (see Subsection~\ref{chowso}), the primitive elements of $\CH^*(\mathrm{SO}_{m};\,\mathbb F_2[v_n])$ are of the form  
$\sum_{i=1}^{s} P_i\cdot e_i$ where $P_i\in \mathbb F_2[v_n]$. The intersection of this set with the ideal $(e_s)$ is just $\mathbb F_2 [v_n]\cdot e_s$. However, by Lemma~\ref{es} the image $y$ of $e_s$ under $\rho$ is not a $v_n^{\mathbb Z}$-torsion, i.e., $\rho(P_s \cdot e_s)$ is zero if and only if $P_s=0$.

Now the injectivity of $\rho_{m}$ follows from Lemma~\ref{mimo} below with $K=\mathbb F_2[v_n]$, $A=\mathrm{CH}^{*}(\mathrm{SO}_{m};\,\mathbb F_2[v_n])$ and $A'=\mathrm{gr}^*_\tau\mathrm{CK}(n)^*(\mathrm{SO}_{m})$ with grading coming from the topological filtration.
\end{proof}

\begin{lm}
\label{mimo}
Let $\varphi\colon A\twoheadrightarrow A'$ be a surjective homomorphism of non-negatively graded Hopf algebras over a ring $K$ with $A^0=K=(A')^0$. Assume that $\mathrm{Ker}\,\varphi$ does not contain primitive elements. Then $\varphi$ is bijective.
\end{lm}
\begin{proof}
We can repeat the proof of~\cite[Proposition~3.9]{MilnorMoore} using the surjectivity of $\varphi$ instead of the flatness of the target Hopf algebra $A'$ over $K$.

We will show by induction that
$$
A^k\cong(A')^k
$$
via $\varphi$ for all $k>0$. Let us denote $J=A^{>0}$ and $J'=(A')^{>0}$. Then
$$
\xymatrix{
P=\mathrm{Ker}\,\big(J\ar[rrr]^{\,a\,\mapsto\,\Delta(a)-1\otimes a-a\otimes1}&&&J\otimes_{K}J\big)
}
$$
obviously coincides with the set of primitive elements of $A$, and similarly
$$
\xymatrix{
P'=\mathrm{Ker}\,\big(J'\ar[rrr]^{\,a\,\mapsto\,\Delta(a)-1\otimes a-a\otimes1}&&&J'\otimes_{K}J'\big)
}
$$
coincides with the set of primitive elements of $A'$. Since 
$$
\big(J\otimes_{K}J\big)^k=\bigoplus_{i=1}^{k-1}A^i\otimes_{K}A^{k-i}
\quad\text{ and }\quad
\big(J'\otimes_{K}J'\big)^k=\bigoplus_{i=1}^{k-1}(A')^i\otimes_{K}(A')^{k-i},
$$
and $A^{i}\cong(A')^{i}$ for $i<k$ by induction hypothesis, we obtain that 
$$
(\varphi\otimes\varphi)^k\colon (J\otimes_{K}J)^k\rightarrow (J'\otimes_{K}J')^k
$$
is bijective. Now we can deduce the injectivity of $\varphi^k$ from the diagram
$$
\xymatrix{
P^k\ar[r]\ar@{^{(}->}[d]&J^k\ar[r]\ar[d]&(J\otimes_{K}J)^k\ar[d]^{\cong}\\
(P')^k\ar[r]&(J')^k\ar[r]&(J'\otimes_{K}J')^k.
}
$$
\end{proof}

\begin{Cr}
\label{free}
For $n\geq1$, $m\le 2^{n+1}$ the group $\mathrm{CK}(n)^*(\mathrm{SO}_m)$ is a finitely generated free $\F_2[v_n]$-module.
\end{Cr}
\begin{proof}
This follows, e.g., from Proposition~\ref{onlytorsion}\,{\it i}{\rm)}.
\end{proof}

\subsection{Computation of the connective Morava K-theory}

In this section we compute the connective Morava K-theory of 
$\mathrm{SO}_m$ {\it as an algebra} for small $m$.

\begin{Th}
\label{connective}
For $n\geq1$, $m\leq2^{n+1}$ the algebra $\ckn\so$ has the following structure:
$$\ckn\so = \F_2[v_n] [\widetilde e_1, \widetilde e_3, \ldots, \widetilde e_{2r-1}]/(\widetilde e_{2i-1}^{\,\,2^{k_i}}),$$
where $r=\lfloor\frac{m+1}{4}\rfloor$, $k_i=\lfloor\mathrm{log}_2(\frac{m-1}{2i-1})\rfloor$, and, moreover, all $\,\widetilde e_{2i-1}$ map to $e_{2i-1}$ in the Chow ring.

Furthermore, $\widetilde{e}_1$ can be chosen as $c_1^{\mathrm{CK}(n)}(e^{int}_1)$,
where $e^{int}_1$ is a lift of ${e_1\in\CH^1(\mathrm{SO}_m; \mathbb{F}_2)}$ to $\CH^1(\mathrm{SO}_m; \mathbb{Z})\cong \mathrm{Pic}(\mathrm{SO}_m)$.
\end{Th}

The proof is just an application of Proposition~\ref{abstracthopf} to $\ckn\so$.
To ensure the assumptions we need the following results.

\begin{Lm}
\label{techclaimlemma}
Take $n\geq1$ and suppose that $m\le 2^{n+1}$.
Then for all ${0<i,\,j\leq\lfloor\frac{m+1}{4}\rfloor}$  
the equation
\begin{align}
\label{techclaim}
(1-2^n)x+(2j-1)2^y=(2i-1)2^{k_i}
\end{align}
with $k_h=\lfloor\mathrm{log}_2(\frac{m-1}{2h-1})\rfloor$ for $h=i,j$ has no integral solutions $x>0$, $0\le y<k_j$.
\end{Lm}
\begin{proof}
Since $(2j-1)2^{k_j} \le 2^{n+1}-1$ and $y<k_j$, we can conclude that ${(2j-1)2^y\leq2(2^n-1)}$. Therefore, for $x\geq2$ we get a contradiction. In the case $x=1$ we conclude that $y$ cannot be positive, because the left hand side of~(\ref{techclaim}) would be odd, while the right hand side would be even, neither $y$ can equal $0$, because the left hand side of~(\ref{techclaim}) would be non-positive, while the right hand side would be positive.
\end{proof}

\begin{Lm}
\label{delta-picard}
In the notation of Theorem~\ref{connective} consider $\widetilde{e}_1= c_1^{\mathrm{CK}(n)}(e^{int}_1)$ in $\ckn\so$. 
Then $\Delta(\widetilde{e}_1) = \sum a_{ij}\, \widetilde{e}_1^{\,i} \otimes \widetilde{e}_1^{\,j}$ where
$\sum a_{ij\,}x^iy^j$ is the formal group law of ${\mathrm{CK}(n)^*}$.
\end{Lm}
\begin{proof}
The class $e_1^{int}\in \CH^1(G)\cong\mathrm{Pic}(G)$ can be identified with the class of a line bundle $L$ on $G$.
Denote by $m\colon G\times G\rarr G$ the multiplication.
The line bundle $m^*L$ on $G\times G$ is isomorphic to $p_1^*(L)\otimes p_2^*(L)$
as can be seen, e.g., from the primitivity of $e_1^{int}$ in the Hopf algebra $\CH^*(G)$ (this follows
from the grading reasons).

However, the first Chern class commutes with pullbacks, 
therefore,
$$
\Delta(\widetilde e_1)=m^{\mathrm{CK}(n)}(\widetilde e_1) = c_1^{\mathrm{CK}(n)}\big( m^*(e^{int}_1)\big)=c_1^{\mathrm{CK}(n)}\big(p_1^*(L)\otimes p_2^*(L)\big).
$$
Thus, by definition of the formal group law for $\mathrm{CK}(n)$ we can 
conclude that $\Delta(\widetilde e_1)$ equals
$\sum a_{ij}\, p_1^*\big(c_1^{\mathrm{CK}(n)}(L)\big)^i p_2^* 
\big(c_1^{\mathrm{CK}(n)}(L)\big)^j.$
\end{proof}

This finishes the proof of Theorem~\ref{connective}, and combining it with Theorem~\ref{stab} we obtain the following result.

\begin{tm}
\label{answer}
For the group $\mathrm{SO}_m$ with $m\leq2^{n+1}$, $n\geq1$, the ring $\mathrm K(n)^*(\mathrm{SO}_m)$ is non-canonically isomorphic to the ring $\mathrm{CH}^*(\mathrm{SO}_m;\,\mathbb F_2[v_n^{\pm1}])$. For $m\geq 2^{n+1}+1$, the canonical map $$\mathrm K(n)^*(\mathrm{SO}_m)\rightarrow\mathrm K(n)^*(\mathrm{SO}_{m-2})$$ is an isomorphism.
\end{tm}

As an immediate corollary of this theorem and \cite[Theorem~5.7]{PShopf} we obtain the following statement. 

\begin{cl}\label{cormax}
Let $X$ be a connected component of the maximal orthogonal Grassmannian for a generic $m$-dimensional quadratic form with trivial discriminant.

Then there exists a $\mathrm K(n)$-motive $\mathcal R$ such that the $\mathrm K(n)$-motive of $X$ decomposes as:
\begin{equation}\label{decompGB}
\mathcal{M}_{\mathrm K(n)}(X)\simeq\bigoplus_{i\in\mathcal{I}}\mathcal{R}\{i\}
\end{equation}
for some multiset of integers $\mathcal{I}$.

Moreover, over a splitting field of $X$ the rank of $\mathrm K(n)^*(\overline{\mathcal R})$ equals the rank of $\mathrm K(n)^*(\mathrm{SO}_m)$. In particular, for $m> 2^{n+1}$ the $\mathrm K(n)$-motive of $X$ is always non-trivially decomposable. {\rm(}Note that contrary to this by a result of Vishik the Chow motive of $X$ is always indecomposable{\rm)}. 
\end{cl}

\begin{proof}
Note that~\cite[Theorem~5.7]{PShopf} is formulated for the variety of Borel subgroups. Nevertheless, since the variety $X$ 
is generically split, it has the same type of a motivic decomposition, and~\cite{PShopf} can be applied. Note also that $H^*$ in~\cite[Theorem~5.7]{PShopf} equals $\mathrm K(n)^*(\mathrm{SO}_m)$, since $X$ corresponds to a  generic quadratic form (see~\cite[Example~4.7]{PShopf}).

Now a direct application of~\cite[Theorem~5.7]{PShopf} implies decomposition~\eqref{decompGB}, where $\mathcal R$ is a motive such that $\mathrm K(n)^*(\overline{\mathcal R})\simeq\mathrm K(n)^*(\mathrm{SO}_m)$ as graded $\F_2[v_n^{\pm 1}]$-modules.

It remains to compare the ranks of $\mathrm K(n)^*(\mathrm{SO}_m)$ and $\mathrm K(n)^*(\overline X)$ when $m>2^{n+1}$. By Theorem~\ref{answer} the rank of $\mathrm K(n)^*(\mathrm{SO}_m)$ equals $2^{2^n-1}$. On the other hand, the rank of $\mathrm K(n)^*\left(\overline X\right)$ equals $2^{\lfloor(m-1)/2\rfloor}$ which is strictly greater than $2^{2^n-1}$. In particular, there are at least two summands in the motivic decomposition~\eqref{decompGB}.
\end{proof}

As another corollary of Theorem~\ref{answer} we also obtain a similar result about spin groups.

\begin{cl}
\label{kn-spin}
For the group $\mathrm{Spin}_m$ with $m\leq2^{n+1}$, $n\geq1$, the ring $\mathrm K(n)^*(\mathrm{Spin}_m)$ is non-canonically isomorphic to the ring $\mathrm{CH}^*(\mathrm{Spin}_m;\,\mathbb F_2[v_n^{\pm1}])$. For $m\geq 2^{n+1}+1$ the canonical map $$\mathrm K(n)^*(\mathrm{Spin}_m)\rightarrow\mathrm K(n)^*(\mathrm{Spin}_{m-2})$$ is an isomorphism.
\end{cl}
\begin{proof}
It follows from the diagram
$$
\xymatrix{
\mathrm K(n)(\mathrm BT_{\mathrm{SO}_m})\ar[r]^{\!\!\mathfrak c}\ar[d]&\mathrm K(n)(\mathrm{SO}_m/B)\ar[r]\ar@{=}[d]&\mathrm K(n)(\mathrm{SO}_m)\ar[d]\\
\mathrm K(n)(\mathrm BT_{\mathrm{Spin}_m})\ar[r]^{\!\!\!\mathfrak c}&\mathrm K(n)(\mathrm{Spin}_m/B)\ar[r]&\mathrm K(n)(\mathrm{Spin}_m)
}
$$
(see Subsection~\ref{characteristic}) that $\mathrm K(n)(\mathrm{Spin}_m)$ coincides with the quotient of $\mathrm K(n)(\mathrm{SO}_m)$ modulo the ideal generated by the image of $\mathfrak c\big(\tau^1\mathrm K(n)(\mathrm BT_{\mathrm{Spin}_m})\big)$ in $\mathrm K(n)(\mathrm{SO}_m)$. Obviously, the ideal generated by $\mathfrak c\big(\tau^1\mathrm K(n)(\mathrm BT_{\mathrm{Spin}_m})\big)$ coincides with the ideal generated by $\mathfrak c(\varpi_l)$ and $\mathfrak c\big(\tau^1\mathrm K(n)(\mathrm BT_{\mathrm{SO}_m})\big)$, where $\varpi_l$ denotes the last fundamental weight. Our choice of $\widetilde e_1$ in Theorem~\ref{connective} guaranties that $\mathfrak c(\varpi_l)$ goes to ${e_1\in\mathrm K(n)(\mathrm{SO}_m)}$, therefore,
$$
\mathrm K(n)(\mathrm{Spin}_m)=\mathrm K(n)(\mathrm{SO}_m)/e_1=\mathbb F_2[v_n^{\pm1}][e_{3},\ldots,e_{2r-1}]/(e_{2i-1}^{2^{k_i}}),
$$
where $r=\lfloor\frac{m+1}{4}\rfloor$, $k_i=\lfloor\mathrm{log}_2(\frac{m-1}{2i-1})\rfloor$ for $m\leq2^{n+1}$, and $r=2^{n-1}$, $k_i=\lfloor\mathrm{log}_2(\frac{2^{n+1}-1}{2i-1})\rfloor$ for $m\geq2^{n+1}+1$. Now, the stabilization is clear, and the result for small $m$ follows from the analogous result on $\mathrm{CH}^*(\mathrm{Spin}_m)$, see~\cite[Table~II]{Kac}.
\end{proof}

\section{Application: Morava K-theory motives of projective quadrics}
\label{nikita}

\subsection{Rational projectors} 
\label{sec-heightn}

As before we denote by $\mathrm K(n)^*$ the Morava K-theory with the coefficient ring $\mathbb F_2[v_n^{\pm1}]$. In this section we always assume that $n\geq2$. Our objective is to describe a decomposition of the $\mathrm K(n)$-motive $\mathcal M_{\mathrm K(n)}(Q)$ of a {\it generic} quadric $Q$ into indecomposable summands. We start with the following proposition.

\begin{prop}
\label{isdec}
Let $Q$ be a smooth projective quadric {\rm(}not necessarily generic{\rm)} of dimension $D\geq2^n-1$, $n>1$. Then its $\mathrm K(n)$-motive $\mathcal M(Q)$ decomposes as a direct sum of $D-2^n+2$ Tate motives and a motive $\mathcal N$ of rank $2^n$ for $D$ even or $2^n-1$ for $D$ odd. 
\end{prop}
\begin{proof}
Let us denote $D'=D-2^n+1$. Then using Proposition~\ref{mod2} we conclude that $\pi_i=v_n^{-1}\,h^i\times h^{D'-i}$ for $0\leq i\leq D'$ is a system of $D'+1$ rational orthogonal projectors corresponding to Tate summands.
\end{proof}

We will show that the remaining summand $\mathcal N$ is in general indecom\-posable. For a split quadric $Q$ we can give an explicit decomposition of $\mathcal N$ into Tate summands in terms of orthogonal projectors. Consider the basis $h^i$, $l_i$, $0\leq i\leq d$ of the ring $\mathrm K(n)^*(Q)$, where $D=\mathrm{dim}\,Q$ is equal to $2d$ or $2d+1$ defined in Proposition~\ref{table}, and let us denote $h^i=0=l_i$ for $i<0$.

\begin{prop}
\label{diagonal}
We use the following notation: 
\begin{align*}
D'&=D-2^n+1,\\
d'&=D'-d
\end{align*}
for $n>1$. Then the following statements hold.
\begin{enumerate}[{\rm (i)}]
\item
The diagonal $\Delta\in\mathrm K(n)^*(Q\times Q)$ is equal to
\begin{align*}
\Delta=\sum_{i=0}^d\big(h^i\times l_i+l_i\times h^i\big)+v_n\sum_{i=d'}^dl_i\times l_{D'-i}\,+\\+\delta_{\,0,\,D\ \!\mathrm{mod}\,4}\cdot (h^d+v_nl_{d'})\times(h^d+v_nl_{d'}),
\end{align*}
where $\delta_{\,0,\,D\ \!\mathrm{mod}\,4}=\begin{cases}0,&\text{ if }D\not\equiv0\!\mod4;\\1,&\text{ if }D\equiv0\!\mod4.\end{cases}$ 
\item
The projectors $\pi_i=v_n^{-1}\cdot h^i\times h^{D'-i}$ for $0\leq i\leq D'$ together with
$$
\varpi_j=(h^j+v_nl_{D'-j})\times(l_j+v_n^{-1}h^{D'-j})
$$ 
for $d'\leq j\leq d-1$ and
$$
\varpi_d=\big(h^d+v_nl_{d'}\big)\times\big(l_d+v_n^{-1}h^{d'}+\delta_{\,0,\,D\ \!\mathrm{mod}\,4}\cdot(h^d+v_nl_{d'})\big)
$$
define a decomposition of $\Delta$ into a sum of $2d+2$ orthogonal Tate motives. Observe that $\pi_i=h^i\times l_{i}$ and $\pi_{D'-i}=l_i\times h^i$ for $i<d'$.
\end{enumerate}
\end{prop}

The proof is straightforward. 

\begin{xm}
Consider the case $n=2$ and $D=3$. Then $D'=0$, $d'=-1$, and $\Delta$ is given by
$$
\Delta=1\times l_0+l_0\times1+h\times l_1+l_1\times h+v_2\cdot l_0\times l_0.
$$
The summands $l_{-1}\times l_1$ and $l_{1}\times l_{-1}$ denote zero according to our convention.
\end{xm}

\subsection{Computation of the co-action}

We denote by $G$ the split group $\mathrm{SO}_m$ and by $P_1$ the maximal parabolic subgroup corresponding to the first simple root in the Dynkin diagram. Let $Q=G/P_1$ denote the {\it split} quadric of dimension $m-2=2d$ or $m-2=2d+1$.

We will describe the co-action
$$
\rho\colon \mathrm K(n)^*(Q)\rightarrow \mathrm K(n)^*(G)\otimes_{\mathbb F_2[v_n^{\pm1}]}\mathrm K(n)^*(Q)
$$
of $\mathrm K(n)^*(G)$ on $\mathrm K(n)^*(Q)$, see Section~\ref{torsor}.  Observe that since $\mathrm K(n)^*(Q)$ is generated as an algebra by the elements $h$ and $l=l_d$ in the notation of Proposition~\ref{table} and since $\rho(h)=1\otimes h$ by~\cite[Lemma~4.12]{PShopf}, the co-action is determined by $\rho(l)$.

It is convenient here to work with the {\it connective} Morava K-theory $\mathrm{CK}(n)^*$. Sending $v_n$ to $0$ we obtain a surjective map from the connective Morava K-theory onto the Chow theory. 

Assume that $m\leq2^{n+1}$, $n>1$ and recall that 
$$
\mathrm{CK}^*(n)(\mathrm{SO}_m)\cong\mathbb F_2[v_n][e_1,\ldots,e_{\lfloor\frac{m-1}{2}\rfloor}]/(e_i^2=e_{2i}),
$$
 where $\mathrm{codim}\,e_i=i$ by Theorem~\ref{connective}, and denote $s=\left\lfloor\frac{m-1}{2}\right\rfloor$ and $r=\left\lfloor\frac{m+1}{4}\right\rfloor$. 
 
 As in the previous section by an {\it index set} we mean an $r$-tuple $I=(d_1,\ldots,d_r)$ with $0\leq d_i<2^{k_i}$ and $k_i=\lfloor\mathrm{log}_2(\frac{m-1}{2i-1})\rfloor$. For an index set $I$ let $e_I$ denote $\prod_{i=1}^re_{2i-1}^{d_i}$, thus, $\{e_I\mid I\text{ an index set\,}\}$ is an $\mathbb F_2[v_n]$-basis of $\mathrm{CK}^*(n)(\mathrm{SO}_m)$ (cf. Theorem~\ref{connective}). 
\begin{lm}
\label{co-action2}
Assume that $m\leq 2^n$, $n>1$. Then the co-action of $\mathrm{K}(n)^*(G)$ on $\mathrm{K}(n)^*(Q)$ is given by the equation
$$
\rho(l)=\sum_{i=1}^{s}e_i\otimes h^{s-i}+1\otimes l +v_n\sum_{} e_I\otimes q_I
$$
for some $q_I\in\mathrm{K}(n)^*(Q)$, where the last sum is taken over $e_I\neq e_i$ for $1\leq i\leq s$.
\end{lm}
\begin{proof}
It follows from~\cite[Lemma~7.2]{PShopf} that
$$
\rho_{\mathrm{CK}(n)}(l)=\sum_{i=1}^{s}e_i\otimes h^{s-i}+1\otimes l+v_n\sum e_I\otimes q_I
$$
for some homogeneous $q_I\in\mathrm{CK}(n)^*(Q)\setminus0$. Observe that the codimension of $q_I$ is at most $m-2$. Since $\rho_{\mathrm{CK}(n)}(l)$ is homogeneous of codimension $s$, none of such $e_I$ can be equal to $e_{2i-1}^{2^x}$ for some $x\leq{k_i}-1$ by dimensional reasons. Now the result about $\mathrm K(n)^*$ follows.
\end{proof}

\begin{lm}
\label{co-action}
Let $Q=Q_{2^n-1}$ be a split $(2^n-1)$-dimensional quadric, $n>1$.
The co-action of $\mathrm{K}(n)^*(\mathrm{SO}_{2^n+1})$ on $\mathrm{K}(n)^*(Q_{2^n-1})$ is given by
\begin{align*}
\rho(l)=\sum_{i=1}^{2^{n-1}}e_i\otimes h^{2^{n-1}-i}+1\otimes l+v_n\,e_{2^{n-1}}\otimes l_0+v_n\sum_{}e_I\otimes q_{I}
\end{align*}
for some $q_I\in\mathrm{K}(n)^*(Q)$, where the last sum is taken with $e_I\neq e_i$ for all $1\leq i\leq s$.
\end{lm}
\begin{proof}
As in Lemma~\ref{co-action2} we have that
$$
\rho_{\mathrm{CK}(n)}(l)=\sum_{i=1}^{2^{n-1}} e_i\otimes h^{2^{n-1}-i}+1\otimes l+v_n\sum e_I\otimes q_I
$$
for some $q_I\in\mathrm{CK}(n)(Q)\setminus0$. Since $\rho_{\mathrm{CK}(n)}(l)$ is homogeneous of codimension $2^{n-1}$, by dimensional reasons there exists only one possibility
when $e_I$ in the sum above is equal to some $e_s$, more precisely, the case $e_I= e_{2^{n-1}}$ cannot be excluded. Moreover, we will show that such a summand indeed appears.

So, the $\mathrm K(n)$-co-action is determined by the formula
\begin{align*}
\rho(l)=\sum_{i=1}^{2^{n-1}}e_i\otimes h^{2^{n-1}-i}+1\otimes l+c\cdot v_n\,e_{2^{n-1}}\otimes l_0+v_n\sum_{|I|>1}e_I\otimes q_{I}
\end{align*}
for some $c\in\mathbb F_2$ and $q_{I}\in\mathrm K(n)^*(Q)$. We claim that $c=1$.

Let $E$ be a generic $\mathrm{SO}_{2^n+1}$-torsor. Recall that a motivic decomposition of $E/P_1$ provides a decompo\-sition of $A^*(G/P_1)$ into a sum of $A^*(G)$-comodules. 

Since $\pi=v_n^{-1}\cdot1\otimes1$ is a projector in $\mathrm K(n)^*(Q_{2^n-1})\otimes\mathrm K(n)^*(Q_{2^n-1})$, and the diagonal is equal to
$$
\Delta=\sum_{i=0}^{2^{n-1}-1}(h^i\otimes l_i+l_i\otimes h^i)+v_n\cdot l_0\otimes l_0,
$$
see Proposition~\ref{diagonal}, the projector
$$
\Delta-\pi=\sum_{i=1}^{2^{n-1}-1}(h^i\otimes l_i+l_i\otimes h^i)+(1+v_nl_0)\otimes(v_n^{-1}+l_0)
$$
defines a summand of $\mathcal M_{\mathrm K(n)}(E/P_1)$. Then using~\cite[Theorem~4.14]{PShopf} and~\cite[Remark~4.15]{PShopf} we conclude that the realization functor maps the motivic summands $\mathcal M_1=(\mathcal M(E/P),\,\pi)$ and $\mathcal M_2=(\mathcal M(E/P),\,\Delta-\pi)$ to the $\mathrm K(n)^*(\mathrm{SO}_{2^n+1})$-subcomodules $C_1=\mathrm K(n)^*(\overline{\mathcal M_1})=\mathbb F_2[v_n^{\pm1}]\cdot 1$ and 
$$
C_2=\mathrm K(n)^*(\overline{\mathcal M_2})=\mathbb F_2[v_n^{\pm1}]\cdot(1+v_nl_0)\oplus\bigoplus_{i=1}^{2^{n-1}-1}(\mathbb F_2[v_n^{\pm1}]\cdot h^i\oplus\mathbb F_2[v_n^{\pm1}]\cdot l_i)
$$
of $\mathrm K(n)^*(Q)$, respectively (here $\,\overline{\phantom{a\,}}\,$ denotes the restriction map $\mathrm{res}_{\,\overline k/k}$,  see Subsection~\ref{torsor}). Since $\rho(l)$ belongs to the submodule $\mathrm K(n)^*(\mathrm{SO}_{2^n+1})\otimes C_2$, it remains to collect the terms of $\rho(l)$ lying in  $\mathbb F_2[v_n^{\pm1}]e_{2^{n-1}}\otimes\mathrm K(n)^*(Q)$ to conclude that $c=1$.
\end{proof}

\subsection{Motive of a generic quadric}

We proved in Proposition~\ref{isdec} that the $\mathrm K(n)$-motive ($n>1$) of any smooth projective quadric of dimension $D\geq 2^n-1$ has $D+2-2^n$ Tate summands. In the present subsection we show that the remaining summand of a {\it generic} quadric is indecomposable. This reproves, in particular, that the Chow motive of a generic quadric is indecomposable, see~\cite{Vlect,Ksuff}. As before $G$ denotes $\mathrm{SO}_m$, $P_1$ denotes the first maximal parabolic subgroup, and $E$ denotes a generic torsor.

\begin{lm}
\label{ind}
For $n>1$ the $\mathrm K(n)$-motive $\mathcal M$ of a generic quadric $E/P_1$ of dimension $D=m-2$ such that $0<D\leq2^n-2$ is indecomposable.
\end{lm}
\begin{proof}
Assume that $\mathcal M$ decomposes as a direct sum $\mathcal M=\mathcal M_1\oplus\mathcal M_2$. Then by~\cite[Theorem~4.14]{PShopf} and~\cite[Remark~4.15]{PShopf} the realization $\mathrm K(n)^*(\overline{\mathcal M})$ of the motive $\mathcal M$ over $\overline k$ is a direct sum of the realizations $C_1=\mathrm K(n)^*(\overline{\mathcal M_1})$ and $C_2=\mathrm K(n)^*(\overline{\mathcal M_2})$ of the motives $\mathcal M_1$ and $\mathcal M_2$ over $\overline k$ (see Subsection~\ref{torsor}).

Recall that $s=\lfloor\frac{m-1}{2}\rfloor$ equals $d$ for $D=2d$ even, and $s=d+1$ for $D=2d+1$ odd. Since $C_i$ are graded and the dimension of the quadric is small, we can assume that either $l\in C_1$ or $l+h^d\in C_1$. Since $C_1$ is a sub-comodule, we conclude that $\rho(l)$ or $\rho(l+h^d)$ belongs to $\mathrm K(n)^*(\mathrm{SO}_{2^n+1})\otimes C_1$. Using the description of $\rho(l)$ from Lemma~\ref{co-action2} and collecting the terms lying in $\mathbb F_2[v_n^{\pm1}]e_{1}\otimes\mathrm K(n)^*(Q)$ we conclude that $h^{s-1}\in C_1$.

But similarly, each $l_k$ for $0\leq k\leq d-1$ lies in one of $C_i$'s, and since $C_i$'s are sub-comodules, $\rho(l_k)$ belongs to the submodule $\mathrm K(n)^*(\mathrm{SO}_{2^n+1})\otimes C_i$. Then the description of $\rho(l_k)$ from Lemma~\ref{co-action2} implies that $h^{s-1}\in C_i$. However, the sum is direct, so that all $l_k$ lie in $C_1$. 

Next, using the description of $\rho(l)$ we conclude that all $h^k$ for $0\leq k\leq s-1$ lie in $C_1$. This finishes the proof for $D$ odd, and for $D$ even it remains to use the description of $\rho(l_{d-1})$ to show that $h^d\in C_1$. Therefore, $C_1=\mathrm K(n)^*(\overline{\mathcal M})$, $C_2=0$, and $\mathcal M$ is indecomposable.
\end{proof}

\begin{lm}
\label{ind2}
The $\mathrm K(n)$-motive $$\mathcal M=(\mathcal M(E/P_1),\,\Delta-v_n^{-1}\cdot1\otimes1)$$ for a generic $(2^n-1)$-dimensional quadric is indecomposable $($for $n>1)$.
\end{lm}
\begin{proof}
The proof of Lemma~\ref{ind} can be repeated with the use of Lemma~\ref{co-action} instead of Lemma~\ref{co-action2}.
\end{proof}

\begin{lm}
\label{indec}
A generic odd-dimensional projective quadric $E/P_1$ of any dimension $\geq2^n-1$ has an indecomposable $\mathrm K(n)$-motivic summand of rank $2^n-1$. A generic even-dimensional projective quadric $E/P_1$ of any dimension $\geq2^n-2$ has an indecomposable $\mathrm K(n)$-motivic summand of rank $2^n$ $($for $n>1)$.
\end{lm}
\begin{proof}
Let $N=2^n-1$ or $N=2^n$ depending on the dimension of the quadric and assume that all indecomposable summands of the $\mathrm K(n)$-motive of $E/P_1$ have rank $<N$. We can assume that $E/P_1$ with this property has the least possible dimension $\mathrm{dim}\,(E/P_1)\geq N$. By Lemmata~\ref{ind} and~\ref{ind2}, a generic quadric of dimension up to $2^n-2$ or $2^n-1$ over any field of characteristic $0$ has an indecomposable summand of rank $N$, in particular, $\mathrm{dim}\,(E/P_1)>2^n-1$.

Let $C$ denote the commutator subgroup $[L_1,L_1]$ of the Levi subgroup $L_1$ of the parabolic subgroup $P=P_1$ in $G$ and let $F$ denote a generic $C$-torsor. In fact, $C=\mathrm{SO}_{m-2}$ and we can assume that $F$ is a $C_{L}$-torsor over a field extension $L/k$ such that $F$ and $E_L$ define the same twisted forms of $G$ over $L$. 

Then the {Chow} motive of $\!\,_{F\,}Q_{m-2}$ is isomorphic to a sum of two Tate motives and a (shifted) motive of a {\it generic} projective quadric of dimension $m-4$,
$$
\mathcal M_{\mathrm{CH}}(\!\,_{F\,}Q_{m-2})\cong \mathcal M_{\mathrm{CH}}(\mathrm{pt})\oplus\mathcal M_{\mathrm{CH}}(\!\,_{F\,}(\mathrm{SO}_{m-2}/P_1))\{1\}\oplus\mathcal M_{\mathrm{CH}}(\mathrm{pt})\{m-2\},
$$
cf. the proof of~\cite[Lemma~7.2]{PShopf}. Therefore, the same is true for the cobordism motive of $\!\,_{F\,}Q_{m-2}$ by~\cite{VY} and for the $\mathrm K(n)$-motive. 

However, by our assumptions (and by Lemmata~\ref{ind} and~\ref{ind2}) the motive $\mathcal M_{\mathrm{K}(n)}\big(\!\,_{F\,}(\mathrm{SO}_{m-2}/P_1)\big)$ has an indecomposable summand of rank $N$. Therefore, $\mathcal M_{\mathrm{K}(n)}(\!\,_{F\,}Q_{m-2})$ has an indecomposable summand of rank $N$ as well. 
\end{proof}
 
Now Proposition~\ref{isdec} and Lemmata~\ref{ind} and~\ref{indec} give us in a certain sense upper and lower bounds for the size of an indecomposable summand in the motive of a generic quadric. 
This proves the following

\begin{tm}
\label{result}
Let $Q$ be a generic quadric of positive dimension $D=2d$ or $D=2d+1$. For $n>1$ consider the $n$-th Morava K-theory $\mathrm K(n)^*$. If $D<2^n-1$, then the corresponding motive of $Q$ is indecomposable, and if $D\geq2^n-1$, then motive of $Q$ decomposes into a direct sum of $D+2-2^n$ Tate motives and an indecomposable summand $\mathcal N$ of rank $2^n$ for $D$ even or $2^n-1$ for $D$ odd. 
\end{tm}

We remark that the assumptions $n>1$ on the Morava K-theory $\mathrm K(n)^*$ and $D>0$ on the dimension of the quadric from Theorem~\ref{result} are important. Indeed, the $0$-dimensional quadric $Q_0$ defined by the quadratic form $\varphi$ is just a quadratic field extension $\mathrm{Spec}\,k(\sqrt{\mathrm{disc(\varphi)}})$, and for any $\mathrm{SO}_2$-torsor we obtain a quadric with a trivial discriminant. This implies that $\mathrm K(n)^*(Q_0)$ is isomorphic to ${\mathrm K(n)^*(\mathrm{pt}\sqcup\mathrm{pt})=\mathrm K(n)^*(\mathrm{pt})\oplus\mathrm K(n)^*(\mathrm{pt})}$, in particular, $l=l_0$ is rational and the motive is split. In terms of the co-action we have ${\mathrm K(n)^*(\mathrm{SO_2})=\mathrm K(n)^*(\mathbb G_{\mathrm m})=\mathbb F_2[v_n^{\pm1}]}$ and $\rho(l)=1\times l$. This does not prevent the motive from being decomposable.

In particular, for $\mathrm K(1)$ (and similarly for $\mathrm K^0$) we do not have the base of induction in the proof of Lemma~\ref{indec}, since $2^1-2=0$.

\end{document}